\documentclass[11pt]{article}
\usepackage{amssymb,amsfonts,amsmath,amsthm}%
\usepackage[font=small,format=plain,labelfont=bf,up]{caption}%
\usepackage[a4paper,nohead,ignorefoot,%
twoside,scale=0.85,bindingoffset=1.25cm]%
{geometry}%
\usepackage[]{graphicx}
\usepackage{color}
\newcommand{\texorpdfstring}[2]{#1}   
\newcommand{\url}[1]{#1} 
\usepackage{hyperref}%
\definecolor{gray}{rgb}{0.2,0.2,.2}
\hypersetup{%
  unicode=true,          
  colorlinks=true,       
  linkcolor=gray,          
  citecolor=gray      
}
\newcommand{\bpm}{\begin{pmatrix}}
\newcommand{\epm}{\end{pmatrix}}

\def\longrightharpoonup{
\relbar\joinrel\joinrel\relbar\joinrel\joinrel\relbar\joinrel\joinrel\rightharpoonup}
\newcommand{\xrightharpoonup}[1]{\stackrel{#1}{\longrightharpoonup}}

\newcommand{\fspace}[1]{{\mathsf{#1}}}
\newcommand{\fspaceL}{\fspace{L}}

\newcommand{\Rset}{{\mathbb{R}}}

\newcommand{\Nset}{{\mathbb{N}}}

\newcommand{\oointerval}[2]{(#1,\,#2)}%
\newcommand{\ccinterval}[2]{[#1,\,#2]}%

\newcommand{\even}{{\rm \,even}}

\newlength{\mhpicDwidth}
\newlength{\mhpicDvsep}
\newlength{\mhpicDhsep}

\newlength{\mhpicPwidth}
\newlength{\mhpicPvsep}
\newlength{\mhpicPhsep}

\newlength{\mhpicWhsep}

\newcommand{\pair}[2]{{\left({#1},\,{#2}\right)}}

\newcommand{\at}[1]{{\left({#1}\right)}}
\newcommand{\nat}[1]{(#1)}
\newcommand{\bat}[1]{{\big(#1\big)}}
\newcommand{\Bat}[1]{{\Big(#1\Big)}}

\newcommand{\ul}[1]{\underline{#1}}

\newcommand{\norm}[1]{\left\|{#1}\right\|}

\newcommand{\abs}[1]{\left|{#1}\right|}

\newcommand{\babs}[1]{\big|{#1}\big|}

\newcommand{\dint}[1]{\,\mathrm{d}#1}

\newcommand{\al}{{\alpha}}
\newcommand{\be}{{\beta}}
\newcommand{\ga}{{\gamma}}

\newcommand{\eps}{{\varepsilon}}

\newcommand{\calA}{\mathcal{A}}
\newcommand{\calB}{\mathcal{B}}

\newcommand{\calF}{\mathcal{F}}

\newcommand{\calL}{\mathcal{L}}
\newcommand{\calM}{\mathcal{M}}
\newcommand{\calN}{\mathcal{N}}
\newcommand{\calO}{\mathcal{O}}
\newcommand{\calP}{\mathcal{P}}
\newcommand{\calQ}{\mathcal{Q}}

\newtheorem{theorem}{Theorem}[section]

\newtheorem{lemma}[theorem]{Lemma}
\newtheorem{cor}[theorem]{Corollary}
\newtheorem{proposition}[theorem]{Proposition}
\newtheorem{ass}[theorem]{Assumption}
\newtheorem{result}[theorem]{Main result}
\begin{document}
%
%
%
\title{Korteweg-de Vries waves in peridynamical media}%
\date{May 22, 2023}
\author{%
Michael Herrmann\footnote{Technische Universit\"at Braunschweig, Germany, {\tt michael.herrmann@tu-braunschweig.de}}\and%
Katia Kleine\footnote{Technische Universit\"at Braunschweig, Germany, {\tt k.kleine@tu-braunschweig.de}}%
}
\maketitle
%
%
%
\begin{abstract}
We consider a one-dimensional peridynamical medium and show the existence of solitary waves with small amplitudes and long wavelength. Our proof uses nonlinear Bochner integral operators and characterizes their asymptotic properties in a singular scaling limit.
\end{abstract}
%
%
%
%
%
\setcounter{tocdepth}{3}
\setcounter{secnumdepth}{3}{\scriptsize{\tableofcontents}}
%
%
%
\section{Introduction}\label{sect:intro}
%
Peridynamics is a nonlocal theory which provides an alternative approach to problems in solid \mbox{mechanics} and replaces the partial differential equations of the classical theory by integro-differential equations that do not involve spatial derivatives, see for instance \cite{SL00}. The internal forces between different material points are described by pairwise interactions similar to nonlinear springs and thus there exists, at least in one space dimension, a close connection to discrete atomistic models such as Fermi-Pasta-Ulam-Tsingou chains (FPUT) with nearest neighbor interactions.
\par
Ever since the seminal paper \cite{FPU55} there has been an ongoing interest in the propagation of traveling waves within atomistic or related systems. A concise and very readable summary of both the existing literature and the current state of research can be found in the review article  \cite{V22}. Traveling waves in peridynamical media are studied in \cite{DB06, SL16} numerically and \cite{PV18,HM19a} establish the existence of solitary waves with large amplitudes by means of two different but related variational methods. 
\par
In this paper we show the existence of Korteweg-deVries (KdV) waves with small amplitudes and generalize similar asymptotic results for various types of lattices.  \cite{ZK65} established the existence of KdV waves in FPUT chains using formal asymptotic analysis and the first rigorous existence proof has been given in the first part of the four-part series of papers \cite{FPI99, FPII02, FPIII04, FPIV04}, while the three other parts deal with the nonlinear orbital stability of those waves.  Periodic waves have been studied in \cite{FML14} while \cite{HML16} concerns chains with more than nearest neighbor interactions. More analytical results on the stability problem can be found in \cite{Miz11,Miz13,HW13,XCKV13} and \cite{FM03,CH18} characterize KdV waves in two-dimensional FPUT lattices.
\par
The existence of KdV-type waves has also been proven for dimer chains, in which the masses and/or the spring constants alternate between two values, as well as for mass-in-mass systems, where each particle interacts additionally with an internal resonator. The results in \cite{HW17,FW18,Fav20,FH20,Fav21,FH23} imply under certain generic conditions the existence of 
wave solutions in which an underlying KdV soliton is superimposed by periodic ripples that are either small (micropterons) or extremely small (nanopterons). See also \cite{KVSD13, XKS15, KSX16} for more details, \cite{GSWW19, FGW20, FH21} for numerical stability investigations, and \cite{VSWP16} for a discussion of non-generic cases without tail oscillations
\par
The KdV equation also governs Cauchy problems in atomistic systems \mbox{provided} that initial values are chosen appopriately. \cite{SW00,KP17} show that the FPUT dynamics can be approximated on large time scales by two KdV solutions  traveling in opposite directions and \cite{HW08, HW09, SKSH14} establish similar results that hold for all times but only a special subclass of initial data. Moreover, \cite{GMWZ14} and \cite{MW22} study the KdV limit in chains with periodically varying and random parameters, respectively, while \cite{HW20} concerns mass-in-mass lattices.
%
%
\subsection{Setting of the problem}
%
We consider a spatially one-dimensional continuous and infinitely extended medium whose material points interact pairwise. According to \cite{SL16}, the simplest peridynamical equation of motion  is given by the integro-differential equation
\begin{align}\label{peridynamic_equation_1}
\partial_{\tilde{t}}^2\tilde{u}\pair{\tilde{t}}{\tilde{y}}
=\int_{-\infty}^{\infty}\partial_r
\Phi\pair{\tilde{u}\pair{\tilde{t}}{\tilde{y}+\xi}-\tilde{u}\pair{\tilde{t}}{\tilde{y}}}{\xi}
\dint{\xi}\,,
\end{align}
where $\tilde{u}$ denotes the scalar displacement field describing the position of material point $\tilde{y}$ at time $\tilde{t}$. Moreover, $\xi$ is the bond variable and the interactions are modeled by the force function $\partial_r \Phi$, which stems from the peridynamical potential $\Phi$  and is assumed to satisfy Newton's third law of motion via
\begin{align}\notag
\partial_r \Phi\pair{r}{\xi}=-\partial_r \Phi\pair{-r}{-\xi}\,.
\end{align} 
Thanks to this identity, we can replace (\ref{peridynamic_equation_1}) by the formula
\begin{align}\label{peridynamic_equation_2}
\partial_{\tilde{t}}^2\tilde{u}\pair{\tilde{t}}{\tilde{y}}
=\int_{0}^{\infty}\partial_r
\Phi\pair{\tilde{u}\pair{\tilde{t}}{\tilde{y}+\xi}-\tilde{u}\pair{\tilde{t}}{\tilde{y}}}{\xi}
-\partial_r\Phi\pair{\tilde{u}\pair{\tilde{t}}{\tilde{y}}-\tilde{u}\pair{\tilde{t}}{\tilde{y}-\xi}}{\xi}\dint{\xi}\,
\end{align}
which is more convenient for our purposes as it involves only positive bond variables $\xi>0$. Moreover, it can be viewed as a universal equation for elastic wave propagation in one space dimension and includes many other lattice or PDE models as special or limiting cases, see the discussion in \cite{SL00,HM19a,KK23}. In particular, assuming that all dominant forces originate from the finetely many bonds $\xi\in\{1,\,\hdots,\,M\}$, the integral with respect to $\xi$ can be replaced by a sum and the peridynamical wave equation \eqref{peridynamic_equation_2}
reduces via $\Phi_m\at{r}=\Phi\pair{r}{m}$ to 
\begin{align*}
\partial_{\tilde{t}}^2\tilde{u}\pair{\tilde{t}}{\tilde{y}}
=\sum_{m=1}^M
\Phi_m^\prime\Bat{\tilde{u}\pair{\tilde{t}}{\tilde{y}+m}-\tilde{u}\pair{\tilde{t}}{\tilde{y}}}
-\Phi_m^\prime\Bat{\tilde{u}\pair{\tilde{t}}{\tilde{y}}-\tilde{u}\pair{\tilde{t}}{\tilde{y}-m}}\,.
\end{align*}
This equation describes that any material point interacts with finitely many other points only and the equivalent lattice model coincides for $M=1$ with the well-known FPUT chain.
\par
In this paper, we study traveling waves in peridynamical media. Combining \eqref{peridynamic_equation_2} with the rescaled traveling wave ansatz
\begin{align}\label{scaling_peridynamics}
\tilde{u}\pair{\tilde{t}}{\tilde{y}}=\tilde{U}_{\eps}\at{\tilde{x}}=\eps\,U_\eps\at{x}\,, \qquad\qquad x=\eps\,\tilde{x}=\eps\,\tilde{y}-\eps\, c_\eps\,\tilde{t}
\end{align}
we obtain the nonlinear and nonlocal equation 
\begin{align}
\label{traveling_wave_equation_scaled}
\eps^3 c_\eps^2\,U_\eps''\at{x}
=\int_{0}^{\infty} \partial_r\Phi\pair{\eps\,U_\eps\at{x+\eps\,\xi}-\eps\,U_\eps\at{x}}{\xi}-\partial_r\Phi\pair{\eps\,U_\eps\at{x}-\eps\,U_\eps\at{x-\eps\,\xi}}{\xi}\dint{\xi}\,,
\end{align}
where $\eps$ is an additional scaling parameter. The existence of solutions has been proven in \cite{PV18,HM19a} by means of constrained optimization techniques (using $\eps=1$) but here we are interested in effective formulas for the long-wave length regime $\eps\to0$ (KdV limit), in which the waves have small amplitudes and propagate with near sonic speed as in \eqref{wave_speed_peridynamic}. In view of the known results for nonlinar lattices and PDEs, the existence of KdV-type waves in nonlocal media is not surprising and generally expected. The rigorous proof, however, is more complicated in the peridynamical setting as it involves the additional variable $\xi$ and requires asymptotic estimates for the continuum of nonlinear interaction forces. Of particular importance is the $\xi$-dependence of the linear and the quadratic terms in the Taylor expansion of $\partial_r\Phi\pair{r}{\xi}$ with respect to $r$.
\begin{ass}\label{assumption_1}
The force function  $\partial_r \Phi: \mathbb{R}\times [0,\,\infty) \to \mathbb{R}$ can be written as
\begin{align}\label{forcefunction}
\partial_r\Phi\pair{r}{\xi}&=\alpha\at{\xi}\,r+\beta\at{\xi}\,r^2+\partial_r\psi\pair{r}{\xi}\,,
\end{align}
where the coefficient functions $\alpha$ and $\beta$ are piecewise continuous and positive for all $\xi\in[0,\,\infty)$.  The function $\partial_r\psi$ is continuously differentiable in $r$, continuous in $\xi$, and satisfies $\partial_r\psi\pair{0}{\xi}=0$ as well as
\begin{align}\label{estimate_second_derivative_psi}
\left|\partial_r^2\psi\pair{r}{\xi}\right|\leq \gamma\at{\xi}\,r^2
\end{align}
for any $\xi$ and all $|r|\leq 1$. Moreover, the integrals
\begin{align}\label{integral_alpha}
 \int_0^\infty \alpha\at{\xi}\,\xi^2\dint{\xi}\,, 
\qquad\qquad
\int_0^\infty \alpha\at{\xi}\, \xi^4\dint{\xi}\,,
\qquad\qquad
 \int_0^\infty \alpha\at{\xi}\, \xi^6\dint{\xi}
\end{align}
and
\begin{align}\label{integral_beta}
\int_0^\infty \beta\at{\xi}\,\xi^3\dint{\xi}\,,
\qquad\qquad
\int_0^\infty \beta\at{\xi}\, \xi^{5/2}\dint{\xi}\,,
 \qquad\qquad
\int_0^\infty \beta\at{\xi}\, \xi^5\dint{\xi}\,,
\end{align}
are positive and finite, while
\begin{align}\label{integral_gamma}
\int_0^\infty \gamma\at{\xi}\, \xi^3\dint{\xi}\,,
\qquad\qquad
\int_0^\infty \gamma\at{\xi}\,\xi^4\dint{\xi}
\end{align}
are well-defined and nonnegative.
\end{ass}
The assumptions on \eqref{integral_alpha}$_1$, \eqref{integral_alpha}$_2$ and \eqref{integral_beta}$_1$ are essential for the asymptotic problem to be well-defined, while the other integrability conditions simplify the analysis and might be weakened at the price of more technical effort. In mechanics one often postulates a finite interaction horizon $H$ such that $\partial_r \Phi\pair{r}{\xi}=0$ holds for all $r$ and $\xi>H$ but our analysis also allows for $H=\infty$ provided that $\al\at\xi$, $\be\at\xi$ and $\ga\at\xi$ decay sufficiently fast for $\xi\to\infty$. We further mention that alternative constitutive laws can be found in the literature. For instance,  the peridynamical forces in \cite{HM19a} are modeled via
\begin{align*}
\Phi\pair{r}{\xi}= a\at\xi\, \Phi_{\text{eff}}\bat{b\at\xi\,r}\,,\qquad \partial_r 
\Phi\pair{r}{\xi}=a\at\xi\, b\at\xi\, \Phi_{\text{eff}}^\prime\bat{b\at\xi\,r}
\end{align*}
in terms of a single effective potential. The choice
\begin{align*}
a\at\xi=\xi\,\chi_{\ccinterval{0}{H}}\at\xi\,,\qquad b\at\xi=\xi^{-1}\,\chi_{\ccinterval{0}{H}}\at\xi\,,\qquad
\Phi_{\text{eff}}\at{s}=
C_2\,s^2+C_3\,s^3\,,
\end{align*}
where $\chi_{\ccinterval{0}{H}}$ denotes the indicator function of the interval $\ccinterval{0}{H}$, implies
\begin{align*}
\al\at\xi=C_2\,\xi^{-1}\,\chi_{\ccinterval{0}{H}}\at\xi\,,\qquad 
\be\at\xi=C_3\,\xi^{-2}\,\chi_{\ccinterval{0}{H}}\at\xi\,,\qquad \psi\pair{r}{\xi}=0
\end{align*}
and is compatible with Assumption  \ref{assumption_1}. Similar constitutive relations have been proposed and studied in \cite{SL16}.
%
%
%
\subsection{Overview on the main result and the proof strategie}
%
Our asymptotic analysis generalizes ideas and methods from \cite{FPI99} and \cite{HML16}, which prove the existence of KdV waves in spatially discrete atomic chains with a single and finitely many bond lenghtes, respectively. However, both the nonlocality and the nondiscreteness of the peridynamical medium necessitate several adjustments, especially the use of Bochner integrals and more careful estimates for the singular limit $\eps\to0$.
\begin{figure}[ht]%
\centering{%
\includegraphics[width=0.95\textwidth]{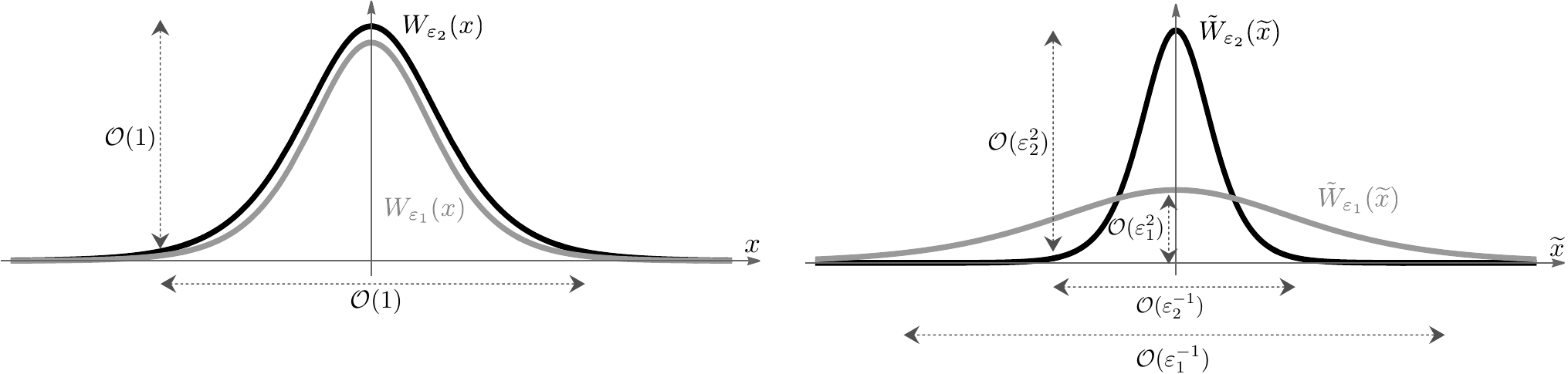}%
}\caption{%
Cartoon of the velocity profile of KdV waves for two different scaling parameters $\eps_1$ (gray) and $\eps_2$ (black) with $0<\eps_1<\eps_2<1$.
The left panel illustrates that the scaled functions $W_\eps$ converge for $\eps\to0$ to the unique limit $W_0$ from \eqref{Kdv_equation_solution} while the right panel shows how the unscaled counterpart $\tilde{W}_\eps$ from \eqref{VelocityScaling} depends on $\tilde{x}$, the original space variable in the comoving frame. 
}%
\label{fig_solutions_W_eps}%
\end{figure}%
\par
As in \cite{HML16}, we link the scaling parameter $\eps$ to the speed via
\begin{align}\label{wave_speed_peridynamic}
c_\eps^2=c_0^2+\eps^2\,, 
\qquad\qquad\qquad
 c_0^2=\int_0^\infty \alpha\at{\xi}\,\xi^2\dint{\xi}\,,
\end{align}
and regard the scaled velocity profile
\begin{align}
\notag
W_\eps:=U'_\eps
\end{align}
as the key quantity, while \cite{FPI99} works with the distance profile $U_\eps\at{\cdot +\eps}-U_\eps\at{\cdot}$ in FPUT chains. Using a convolution operator $\calA_\eta$, which we introduce in (\ref{operator_A_eta}), we can reformulate (\ref{traveling_wave_equation_scaled}) as 
\begin{align}\label{operator_equation_peridynamics}
\calB_\eps W_\eps=\calQ_\eps [W_\eps]+\eps^2\calP_\eps [W_\eps]\,,
\end{align}
where the Bochner integral operator $\calB_\eps$ is given in \eqref{operator_B_eps_peridynamics} and collects all terms that are linear with respect to $W_\eps$. Moreover, the nonlinear Bochner operators $\calQ_\eps$ and $\calP_\eps$  are defined in \eqref{operators_Q_eps_P_eps_peridynamics} and represent all quadratic and higher order terms, respectively. Formal asymptotic arguments applied to \eqref{operator_equation_peridynamics} --- see \S\ref{sect:prelim} as well as \cite{HML16,KK23} for more details --- yield with formula \eqref{operator_equation_B_0_Q_0_peridynamic} an analogue to the KdV traveling wave equation 
\begin{align}\label{KdV_equation_peridynamic}
W_0^{\prime\prime}=d_1\,W_0-d_2\,W_0^2
\end{align}
for the limit $\eps\to0$. The ODE constants $d_1$, $d_2$ depend on the coefficient functions $\alpha$, $\beta$ as described in equation (\ref{KdV_equation_coeffizients}) below and the only even and homoclinic solution is given by
\begin{align}\label{Kdv_equation_solution}
W_0\at{x}=\frac{3\,d_1}{2\,d_2}\operatorname{sech}^2\at{\frac{1}{2}\,\sqrt{d_1}\,x}\in \mathsf{C}^\infty\at{\mathbb{R}}\,.
\end{align}
However, the nonlocal equation \eqref{operator_equation_peridynamics} is not regular but a singular perturbation of \eqref{KdV_equation_peridynamic} and this complicates the analysis for small $\eps$. As in \cite{HML16}, we further introduce the predictor-corrector-ansatz 
\begin{align}\label{predictor_corrector}
W_\eps=W_0+\eps^2\,V_\eps \in\mathsf{L}_{\mathrm{even}}^2\at{\mathbb{R}}\,,
\end{align}
transform the operator equation \eqref{operator_equation_peridynamics} into the equivalent fixed point problem
\begin{align}
\label{FixedPoint}
V_\eps=\calF_\eps[V_\eps]\,,
\end{align}
and employ the Contraction Mapping Principle to prove the existence and local uniqueness of $V_\eps$ for all sufficiently small $\eps$. The key technical problem in this approach is to establish uniform invertibility estimates for a linear but nonlocal operator $\calL_\eps$, which represents the linearization of \eqref{operator_equation_peridynamics} around $W_0$ and whose inverse enters  the definition of the nonlinear operator $\calF_\eps$, see equations \eqref{operator_L_eps_peridynamic} and  \eqref{operator_F_eps_peridynamic}. We further mention that \cite{FPI99}
solves the nonlinar $\eps$-problem for FPUT chains by a variant of the Implicit Function Theorem but also needs careful estimates concerning the inverse of its linearization.
\par
Our main result can be summarized as follows and provides via \eqref{scaling_peridynamics} and 
\begin{align}
\label{VelocityScaling}
\tilde{U}_{\eps}^\prime\at{\tilde{x}}=\tilde{W}_{\eps}\at{\tilde{x}}=\eps^2\, W_\eps\at{x}
\end{align}
a $\eps$-parametrized family of solitary wave solutions to the peridynamical wave equation \eqref{peridynamic_equation_2} that is illustrated in  Figure \ref{fig_solutions_W_eps}.
\begin{result}
Let Assumption \ref{assumption_1} be satisfied and  $\eps>0$ be sufficiently small. Then there exists a unique solution \mbox{$W_\eps\in\mathsf{L}_{\mathrm{even}}^2\at{\mathbb{R}}$} to the scaled peridynamic equation (\ref{traveling_wave_equation_scaled}) with $c_\eps$ as in \eqref{wave_speed_peridynamic} that lies in a small neighborhood of the KdV wave $W_0$ from \eqref{Kdv_equation_solution}. In particular, we have
\begin{align*}
\left\| W_\eps-W_0\right\|_2\leq C\,\eps^2
\end{align*}
for some constant $C>0$ independent of $\eps$.
\end{result}
The paper is organized as follows:  In \S \ref{subsection_integraloperator} we introduce a family of convolution operators which allows us in \S\ref{subsection_operator_equation} to transform the rescaled peridynamical equation (\ref{traveling_wave_equation_scaled}) into the nonlinear integral equation \eqref{operator_equation_peridynamics} for the velocity profile $W_\eps$.  The asymptotic properties of the involved operators are discussed in \S\ref{subsection_operators_properties_1} and \S\ref{subsection_operators_properties_2}.  In \S\ref{subsection_operator_L_eps} we linearize the nonlinear problem (\ref{traveling_wave_equation_scaled}) around $W_0$ and prove in \S\ref{subsection_invertibility_L_eps} that the corresponding linear operator $\calL_\eps$ is uniformly invertible on the space of all even $\mathsf{L}^2$-functions. In \S\ref{subsection_fixed_point_argument} we finally solve the nonlinear problem \eqref{operator_equation_peridynamics} by applying the Contraction Mapping Principle to the corrector equation \eqref{FixedPoint}. Our analysis in \S\ref{sect:prelim} and \S\ref{section_main_result} employs Bochner integrals and we refer to the appendix for more details concerning both the general theory and the operators at hand.
%
%
\section{Preliminaries}\label{sect:prelim}
%
%
In this section we reformulate the rescaled peridynamical equation (\ref{traveling_wave_equation_scaled}) as a nonlinear eigenvalue problem and study the properties of the involved integral operators.
%
%
\subsection{The integral operator \texorpdfstring{$\calA_\eta$}{A}}\label{subsection_integraloperator}
%
%
For $\eta>0$ we denote by $\calA_\eta$ the integral operator 
\begin{align}\label{operator_A_eta}
\left(\mathcal{A}_\eta\,W\right)(x)=\frac{1}{\eta}\int_{x-\eta/2}^{x+\eta/2} W(y)\dint{y} =\frac{1}{\eta}\left(\chi_\eta *W\right)(x)\,,
\end{align} 
which describes the convolution with the indicator function  $\chi_\eta$ of the interval $\left[-\eta/2,\, \eta/2\right]$. Using $\text{sinc}\at{z}=\sum_{j=0}^{\infty}\frac{(-1)^j\, z^{2j}}{\left(2\,j+1\right)!}$ and direct computation we verify
\begin{align*}
\mathcal{A}_\eta\mathtt{e}^{\mathtt{i}kx}
=\text{sinc}\at{\frac{\eta\,k}{2}}\,\mathtt{e}^{\mathtt{i}kx}
=\sum_{j=0}^{\infty}\frac{\eta^{2j}}{2^{2j}\left(2\,j+1\right)!}\,(-1)^jk^{2j}\mathtt{e}^{\mathtt{i}kx}
=\sum_{j=0}^{\infty}\frac{\eta^{2j}}{2^{2j}\left(2\,j+1\right)!}\,\partial_x^{2j}\mathtt{e}^{\mathtt{i}kx}
\end{align*}
for any $k\in\Rset$ and conclude that the pseudo-differential operator $\calA_\eta$ can be regarded as a singular perturbation of the idendity operator $\text{Id}$. In particular, for $\eta:=\eps\,\xi \ll 1$ we obtain the formal expansion
\begin{align}\label{singular_expansion_A_eta_peridynamics}
\calA_{\eps\xi}=\operatorname{Id}+\frac{\eps^2\,\xi^2}{24}\,\partial_x^2+\calO\at{\eps^4}\,,
\end{align}
where the error terms contain higher derivatives. The integral operator $\calA_\eta$ exhibits a number of useful properties which we use throughout the paper.
\begin{lemma}[properties of $\mathcal{A}_\eta$]\label{properties_A_eta}
For all $\eta>0$ the operator $\calA_\eta$ admits the following properties:
\begin{itemize}
\item[1.] 
For $1\leq p\leq\infty$ and $W\in \mathsf{L}^p\at{\mathbb{R}}$ we have $\mathcal{A}_\eta W\in \mathsf{L}^p\at{\mathbb{R}}\cap \mathsf{L}^\infty \at{\mathbb{R}}$ with
\begin{align}\label{estimate_A_eta_1}
\left\| \mathcal{A}_\eta W\right\|_\infty \leq \eta^{-1/p}\, \left\| W\right\|_p\,, \qquad \left\| \mathcal{A}_\eta W\right\|_p \leq \left\| W\right\|_p\,.
\end{align}
\item[2.] 
For any $W\in \mathsf{L}^2\at{\mathbb{R}}$ we have $\left(\mathcal{A}_\eta W\right)\in\mathsf{W}^{1,2}\at{\mathbb{R}}$  with
\begin{align*}
\left(\mathcal{A}_\eta W\right)'\at{x}=\eta^{-1} \left(W\at{x+\eta/2}-W\at{x-\eta/2}\right)\,,\qquad \left\| \left(\mathcal{A}_\eta W\right)'\right\|_2 \leq 2\,\eta^{-1}\, \left\| W\right\|_2
\end{align*}
as well as
\begin{align}\label{asymptotic_A_eta}
\left(\mathcal{A}_\eta W\right)\at{x}\xrightarrow{x\to \pm\infty} 0\,.
\end{align}
\item[3.] 
For any $W\in \mathsf{L}^2\at{\mathbb{R}}$ the  map $\eta\mapsto \calA_\eta W $ is differentiable with derivative 
\begin{align*}
\frac{\dint{}}{\dint{\eta}}(\calA_\eta W )\at{x}=-\eta^{-1}\bat{\calA_\eta W}\at{x}+ (2\eta)^{-1}\left( W\at{x+\eta/2}+W\at{x-\eta/2}\right)\in\mathsf{L}^2\at{\mathbb{R}}\,.
\end{align*}
\item[4.] 
The convex cone $\mathsf{K}:=\left\{ W\in\mathsf{L}^2\at{\mathbb{R}}\,:\,W \text{ is even, nonnegativ and unimodal}\right\}$ is invariant under $\calA_\eta$, where unimodal means that $W$ is monotonically increasing and decreasing for $x<0$ and $x>0$, respectively. 
\item[5.] 
$\calA_\eta$ is a pseudo-differential operator and diagonalizes in Fourier space. Its symbol function is given by
\begin{align}\label{symbol_A_eta}
a_\eta\at{k}:=\operatorname{sinc}\at{k\eta/2}
\end{align}
with $\operatorname{sinc}\at{z}=\sin\at{z}/z$.
\item[6.] 
$\calA_\eta$ is self-adjoint on $\mathsf{L}^2\at{\mathbb{R}}$.
\item[7.] 
For any $W\in\mathsf{L}_{\mathrm{loc}}^1\at{\mathbb{R}}$, the estimate
\begin{align}\label{estimate_A_eta_pointwise}
\left|(\mathcal{A}_\eta W)\at{x}\right|\leq (\mathcal{A}_\eta \left| W\right|)\at{x}
\end{align}
holds pointwise in $x\in\mathbb{R}$. 
\item[8.]  
For any sufficiently regular $W$ we have
\begin{align}\label{error_A_eta_1}
\left\|\mathcal{A}_\eta W-W\right\|_2\leq C\,\eta^2\,\left\| W''\right\|_2\,, \qquad \left\|\mathcal{A}_\eta W-W-\frac{\eta^2}{24}W''\right\|_2\leq C\eta^4\left\| W''''\right\|_2
\end{align}
and
\begin{align*}
\left\|\mathcal{A}_\eta W-W\right\|_\infty\leq C\,\eta^2\,\left\| W''\right\|_\infty\,,\qquad\left\|\mathcal{A}_\eta W-W-\frac{\eta^2}{24}W''\right\|_\infty\leq C\eta^4\left\| W''''\right\|_\infty\,.
\end{align*}
In particular, the convergence
\begin{align}
\label{convergence_A_etaW_to_W}
\calA_\eta W \xrightarrow{\;\eta\to 0\;} W \qquad \text{ strongly in } \quad \mathsf{L}^2\at{\mathbb{R}}
\end{align}
is satisfied for any $W\in\mathsf{L}^2\at{\mathbb{R}}$.
\end{itemize}
\end{lemma}
\begin{proof}
All assertions follow from standard arguments and we refer to \cite[Lemma 2.5]{Her10} and \cite[Lemma 2.3]{HML16} for more details.
\end{proof}
%
%
\subsection{Reformulation of the problem}\label{subsection_operator_equation}
%
%
A key observation is that (\ref{traveling_wave_equation_scaled}) can be reformulated as a nonlinear fixed-point equation that involves the auxiliary operator $\calA_\eta$.
\begin{lemma}\label{lemma_eigenvalueproblem_peridynamics}
For $W_\eps=U_\eps'\in\mathsf{L}^2\at{\mathbb{R}}$ the traveling wave equation (\ref{traveling_wave_equation_scaled}) is equivalent to the nonlinear integral equation
\begin{align}\label{eigenvalueproblem_peridynamics}
\eps^2 c_\eps^2\, W_\eps=\int_0^\infty \xi\, \calA_{\eps\xi}\,\partial_r\Phi\pair{\eps^2\xi\,\calA_{\eps\xi} W_\eps}{\xi}\dint{\xi}\,.
\end{align}
\end{lemma}
\begin{proof}
Let $U_\eps$ be the primitive  of $W_\eps$, such that $U_\eps\at{x}=\int_{x_0}^{x} W_\eps\at{x}\dint{x}$ holds for arbitrary $x_0\in\mathbb{R}$. Using the identities  
\begin{align*}
\eps\left(\,U_\eps\at{x+\eps\,\xi}-U_\eps\at{x}\right)=\eps^2\,\xi\left(\calA_{\eps\xi} U_\eps\right)'\at{x+\eps\xi/2}
=\eps^2\,\xi\left(\calA_{\eps\xi} W_\eps\right)\at{x+\eps\xi/2}
\end{align*}
and
\begin{align*}
\eps\left(\,U_\eps\at{x}-U_\eps\at{x-\eps\,\xi}\right)
=\eps^2\,\xi\left(\calA_{\eps\xi} U_\eps\right)'\at{x-\eps\xi/2}
=\eps^2\,\xi\left(\calA_{\eps\xi} W_\eps\right)\at{x-\eps\xi/2}
\end{align*}
respectively, we obtain (\ref{traveling_wave_equation_scaled}) after differentiating (\ref{eigenvalueproblem_peridynamics}) with respect  to $x$. On the other hand, integrating  (\ref{traveling_wave_equation_scaled}) with respect to $x$ yields (\ref{eigenvalueproblem_peridynamics}) with an additional constant of integration $C$. This, however, must vanish due to $W_\eps\in\mathsf{L}^2\at{\mathbb{R}}$, \eqref{asymptotic_A_eta}, and since $\partial_r \Phi\pair{0}{\xi}=0$ holds for all $\xi$.
\end{proof}
In the next step we transform the eigenvalue problem from Lemma \ref{lemma_eigenvalueproblem_peridynamics} into the operator equation \eqref{operator_equation_peridynamics}. To this end we insert the Taylor expansion (\ref{forcefunction}) and the speed relation (\ref{wave_speed_peridynamic}) into (\ref{eigenvalueproblem_peridynamics}), collect all linear$\,|\,$nonlinear terms on the left$\,|\,$right hand side, and divide  by $\eps^4$. This yields the linear operator $\calB_\eps:\mathsf{L}^2\at{\mathbb{R}}\to\mathsf{L}^2\at{\mathbb{R}}$ with
\begin{align}\label{operator_B_eps_peridynamics}
\calB_\eps\,  W := W+\int_0^\infty \alpha\at{\xi}\,\xi^2\,\frac{ W-\calA_{\eps \xi}^2 W}{\eps^2}\dint{\xi}\,,
\end{align} 
while the nonlinear operators $\calQ_\eps,\, \calP_\eps:\mathsf{L}^2\at{\mathbb{R}}\to\mathsf{L}^2\at{\mathbb{R}}$ are given by
\begin{align}\label{operators_Q_eps_P_eps_peridynamics}
\calQ_\eps [W]:= \int_0^\infty \beta\at{\xi}\,\xi^3\calA_{\eps\xi}\left(\calA_{\eps\xi} W\right)^2\dint{\xi}\,,
\qquad
\calP_\eps [W]:= \frac{1}{\eps^6} \int_0^\infty \xi\,\calA_{\eps\xi}\,\partial_r \psi\pair{\eps^2\xi\,\calA_{\eps\xi}W}{\xi}\dint{\xi}\,.
\end{align}
All these operators are well-defined in the sense of Bochner integrals and the details are given in the appendix, see Proposition \ref{prop:Bochner1}. Of course, \eqref{operator_equation_peridynamics} admits the trivial solution $W_\eps\equiv 0$ but below we show that there also exists another unique solution in a small vicinity of the KdV wave $W_0$.
%
%
\subsection{Properties of the operator \texorpdfstring{$\calB_\eps$}{B}}\label{subsection_operators_properties_1}
%
%
The properties of the pseudo-differential operator $\calB_\eps$ are determined by its Fourier symbol and imply the existence of $\calB_\eps^{-1}$ thanks to the supersonicity of the wave speed \eqref{wave_speed_peridynamic}. The inverse operator is important for proving that the linearized operator $\calL_\eps$ is uniformly invertible, see Proposition \ref{proposition_uniform_invertibility} below. 
\begin{lemma}[properties of $\calB_\eps$]
For any $\eps>0$, the linear operator $\calB_\eps$ is continuous, self-adjoint, uniformly invertible on $\mathsf{L}^2\at{\mathbb{R}}$, and maps the subspace of even functions into itself. Moreover,  in Fourier space it corresponds to the multiplication with the symbol function
\begin{align}\label{symbol_b_eps_peridynamics}
b_\eps\at{k}=1+\int_0^\infty \alpha\at{\xi}\,\xi^2\,\frac{1-\operatorname{sinc}^2\at{\frac{\eps\,\xi\,k}{2}}}{\eps^2}\dint{\xi}\,,
\end{align}
that means we have $\widehat{\calB_\eps V}\at{k}=b_\eps\at{k}\,\widehat{V}\at{k}$
for any $V\in\mathsf{L}^2\at{\mathbb{R}}$ and almost all $k\in\mathbb{R}$ .
\end{lemma}
\begin{proof}
The continuity of $\calB_\eps$ is a direct consequence of the estimates within the proof of Proposition \ref{prop:Bochner1} and the self-adjointness  as well as the invariance of even functions under $\calB_\eps$ follow from the corresponding properties of $\calA_\eta$, see Lemma \ref{properties_A_eta}. Moreover, the existence of $b_\eps$ and the validity of (\ref{symbol_b_eps_peridynamics}) are shown in Proposition \ref{prop:Bochner2}. Since $0\leq 1-\operatorname{sinc}^2\at{y}\leq 1$ holds for all $y\in\mathbb{R}$ we have
\begin{align*}
1\leq b_\eps\at{k}\leq 1+\frac{1}{\eps^2}\,\int_0^\infty \alpha\at{\xi}\,\xi^2\dint{\xi}<\infty
\end{align*}
for any $k\in\mathbb{R}$ due to (\ref{integral_alpha}) and hence $b_\eps^{-1}\at{k}=1/b_\eps\at{k}\leq 1$. This proves the existence of the inverse operator $\calB_\eps^{-1}$ with symbol function  $b_\eps^{-1}$ as well as the uniform bound for its operator norm.
\end{proof}
\begin{figure}[ht]
\centering
\includegraphics[width=0.95\textwidth]{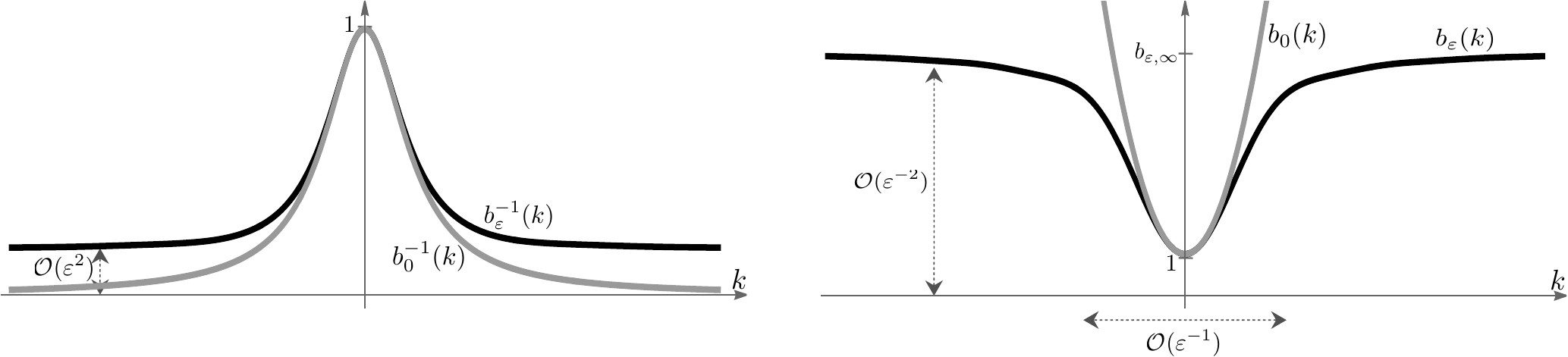}%
\caption{The symbol functions $b_\eps$ (black) from \eqref{symbol_b_eps_peridynamics} and their pointwise limit $b_0$ (gray).%
}
\label{fig_symbol_b_eps_b_0}
\end{figure}
We now study the asymptotic properties of $\calB_\eps$. Inserting the formal expansion (\ref{singular_expansion_A_eta_peridynamics}) into (\ref{operator_B_eps_peridynamics}) and passing to the limit $\eps\to 0$ we obtain 
\begin{align}\notag
\calB_0 W:= W-\left(\frac{1}{12}\int_0^\infty \alpha\at{\xi}\,\xi^4\dint{\xi}\right)\partial_x^2\,W
\end{align}
for any sufficiently smooth function $W$ as well as  $\widehat{\calB_0 W}\at{k}=b_0\at{k}\,\widehat{W}\at{k}$ with
\begin{align}\notag
b_0\at{k}=1+\left(\frac{1}{12}\int_0^\infty \alpha\at{\xi}\,\xi^4\dint{\xi}\right)\,k^2\,.
\end{align}
However, the operator $\calB_0$ is not a regular but a singular limit of $\calB_\eps$ since $b_\eps$ does not converge uniformly to $b_0$ as $\eps\to0$ as illustrated in Figure \ref{fig_symbol_b_eps_b_0}. In particular, we have
\begin{align}\notag
b_\eps\at{k}\xrightarrow{\;k\to\pm\infty\;} 1+\frac{1}{\eps^2}\,\int_0^\infty \alpha\at{\xi}\,\xi^2\dint{\xi}=: b_{\eps, \infty}\,,\qquad b_0\at{k}\xrightarrow{\;k\to\pm\infty\;}\infty\,.
\end{align}
Moreover, the positivity and the quadratic growth of $b_0$ imply that $\calB_0$ is defined on the smaller set $\mathsf{W}^{2,2}\at{\mathbb{R}}$ and admits an inverse with nice smoothing properties while $\calB_\eps^{-1}$ is less regularizing since $b_\eps$ approaches for $k\to\pm\infty$ a constant value of order $\calO\nat{\eps^2}$.
Nonetheless, the operator $\calB_\eps$ is still a sufficiently nice counterpart to $\calB_0$ as shown by the 
following three results.
\begin{lemma}[convergence of $\calB_\eps$]
\label{lemma_convergence_B_eps_to_B_0}
The convergence
\begin{align*}
b_\eps\at{k}\xrightarrow{\eps\to 0} b_0\at{k}
\end{align*} 
holds pointwise for any fixed $k\in\mathbb{R}$. Moroever, for any $W\in \mathsf{W}^{2,2}\at{\mathbb{R}}$ we have
\begin{align*}
\|\calB_\eps W-\calB_0 W\|_2 \xrightarrow{\eps \to 0} 0
\end{align*}
and for $W\in \mathsf{W}^{4,2}\at{\mathbb{R}}$ we even get
\begin{align*}
\|\calB_\eps W-\calB_0 W\|_2 \leq C\,\eps^2\,.
\end{align*}
\end{lemma}
\begin{proof}
Using the Taylor expansion $\operatorname{sinc}^2\at{y}=1-\frac{y^2}{3}+\calO\at{y^4}$
we obtain
\begin{align*}
b_\eps\at{k}
=1+\int_0^\infty \alpha\at{\xi}\,\frac{\xi^4}{12}\,k^2+\calO\at{\eps^2}\dint{\xi}
\xrightarrow{\;\eps\to 0\;} 1+\frac{\int_0^\infty \alpha\at{\xi}\,\xi^4\dint{\xi}}{12}\,k^2=b_0\at{k}
\end{align*}
for all $k\in\mathbb{R}$ and hence the claimed  pointwise convergence. Moreover,  Parseval's Theorem combined with Proposition \ref{prop:Bochner2} implies 
\begin{align}
\label{lemma_convergence_B_eps_to_B_0-e1}
\|\calB_\eps W-\calB_0 W\|_2^2
&= \frac{1}{2\,\pi} \int_{\mathbb{R}} h_\eps\at{k} \dint{k}\,,\qquad h_\eps\at{k}:=|b_\eps\at{k}-b_0\at{k}|^2\,  |\widehat{W}\at{k}|^2\,.
\end{align}
The auxiliary estimate (see Figure \ref{fig_lower_bound_b_eps} for an illustration)
\begin{align*}
0\leq \frac{1}{3}\,y^2 -\left(1-\operatorname{sinc}^2\at{y}\right) \leq \frac{1}{3}\,y^2
\end{align*}
combined with $y=\eps\xi k/2$ ensures
\begin{align}\notag
0\leq b_0\at{k}-b_\eps\at{k}
\leq\frac{1}{12}\,k^2\int_0^\infty \alpha\at{\xi}\,\xi^4\dint{\xi}
\leq C\,\at{1+k^2}
\end{align}
for all $k\in\mathbb{R}$ and we obtain
\begin{align*}
h_\eps\at{k}\leq C \left(1+k^2\right)^2  |\widehat{W}\at{k}|^2\,,\qquad 
\int_{\mathbb{R}} h_\eps\at{k}\dint{k}\leq C\,\|W\|_{2,2}^2\,.
\end{align*}
The second claim is thus a direct consequence of \eqref{lemma_convergence_B_eps_to_B_0-e1}, the pointwise convergence  $b_0\at{k}=\lim_{\eps\to} b_\eps\at{k}$, and the Dominated Convergence Theorem. Finally, let $W\in\mathsf{W}^{4,2}\at{\mathbb{R}}$ be fixed. Proposition \ref{prop:Bochner1} provides
\begin{align*}
\|\calB_\eps W-\calB_0 W\|_2
&\leq\frac{1}{\eps^2}\int_0^\infty \alpha\at{\xi}\,\xi^2\left\|\calA_{\eps\xi}^2 W-W-\frac{\eps^2\,\xi^2}{12}\,W''\right\|_2\dint{\xi} 
\end{align*}
and using
\begin{align*}
\calA_{\eps\xi}^2W-W-\frac{\eps^2\,\xi^2}{12}\,W''
=&\left(\calA_{\eps\xi}+\operatorname{Id}\right) \left(\calA_{\eps\xi} W-W-\frac{\eps^2\,\xi^2}{24}\,W''\right)+\frac{\eps^2\,\xi^2}{24}\left(\calA_{\eps\xi}W''-W''\right)
\end{align*}
we get
\begin{align*}
\left\|\calA_{\eps\xi}^2 W-W-\frac{\eps^2\,\xi^2}{12}\,W''\right\|_2
&\leq 2\,\left\|\calA_{\eps\xi} W-W-\frac{\eps^2\,\xi^2}{24}\,W''\right\|_2
+\frac{\eps^2\,\xi^2}{24}\,\|\calA_{\eps\xi}W''-W''\|_2\\
&\leq 2\,C\,\eps^4\,\xi^4\,\|W''''\|_2+C\,\frac{\eps^4\,\xi^4}{24}\,\|W''''\|_2\\
&\leq C\,\eps^4\,\xi^4
\end{align*}
for all $\xi\in\pair{0}{\infty}$ thanks to (\ref{estimate_A_eta_1})$_2$, where we applied (\ref{error_A_eta_1})$_1$ to $W''$ as well as (\ref{error_A_eta_1})$_2$ to $W$. The third claim now follows in view of \eqref{integral_alpha}.
\end{proof}
\begin{figure}[ht]
\centering
\includegraphics[width=0.95\textwidth]{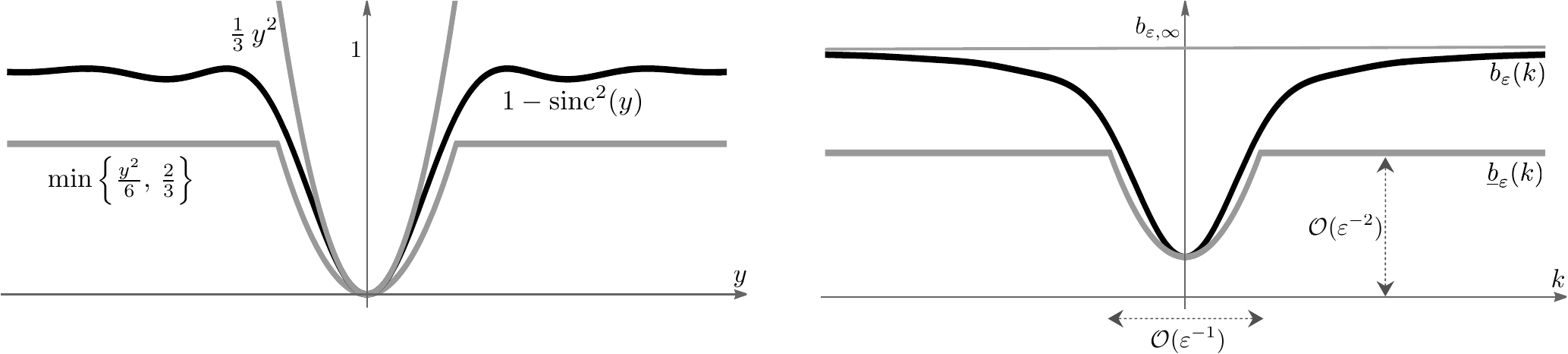}%
\caption{%
On the construction of the pointwise bounds for $b_\eps$ in the proofs of Lemma \ref{lemma_convergence_B_eps_to_B_0} and  Lemma \ref{lemma_lower_bound_b_eps}.
}%
\label{fig_lower_bound_b_eps}
\end{figure}
We next show that  $\calB_\eps^{-1}$ can be written as the sum of an almost compact operator and a small bounded one. A similar result has been derived for atomic chains in \cite{HML16} but the technical details in the peridynamical setting are more involved due to the continuum of bond variables. We start with an auxiliary result that allows us to replace the function $b_\eps$ by a simpler one and refer to the right panel in Figure \ref{fig_lower_bound_b_eps} for an illustration.
\begin{lemma}\label{lemma_lower_bound_b_eps}
For all $k\in\mathbb{R}$ and $\eps>0$ we have
\begin{align*}
b_\eps\at{k}
\geq \underline{b_\eps}\at{k}
:=C_0
 \begin{cases}
1+k^2 & \text{ for } |k|\leq\frac{C_1}{\eps}\,,\\
1+\frac{1}{\eps^2} & \text{ for } |k|>\frac{C_1}{\eps}\,,
\end{cases}
\end{align*}
where the positive constants $C_0$, $C_1$ do not depend on $\eps$.
\end{lemma}
\begin{proof}
We choose $0<h<H<\infty$ such that $\al\at\xi>0$ holds for all $\xi$ with $h\leq \xi\leq H$ and estimate
\begin{align}\notag
b_\eps\at{k}
\geq 1+\int_h^H \alpha\at{\xi}\,\xi^2\,\frac{1-\operatorname{sinc}^2\at{\frac{\eps\,\xi\,k}{2}}}{\eps^2}\, \dint{\xi}
=: b_{\eps,h,H}\at{k}\,.
\end{align} 
Moreover, the properties of the sinus cardinalis (see the left panel in Figure \ref{fig_lower_bound_b_eps}) imply
\begin{align*}
1
\geq 1-\operatorname{sinc}^2\at{\xi\,y}
\geq \min\left\{\frac{\xi^2\,y^2}{6},\, \frac{2}{3}\right\}
\geq \min\left\{\frac{h^2\,y^2}{6},\, \frac{2}{3}\right\}
\end{align*}
and with $y=\eps\,k/2$ we obtain
\begin{align*}
1-\operatorname{sinc}^2\at{\frac{\eps\,\xi\,k}{2}}
\geq \begin{cases}
\frac{\eps^2\,h^2\,k^2}{24} & \text{ for } |k|\leq \frac{4}{\eps\,h}\,,\\
\frac{2}{3}  & \text{ for } |k|> \frac{4}{\eps\,h}\,.
\end{cases}
\end{align*}
The combination of the partial estimates gives
\begin{align*}
b_{\eps, h, H}\at{k}
&\geq C_0
\begin{cases}
1+k^2 &\text{ for } |k|\leq \frac{4}{\eps\,h}\,,\\
1+\frac{1}{\eps^2} &\text{ for } |k|> \frac{4}{\eps\,h}
\end{cases}
\end{align*}
and setting $C_1=4/h$ completes the proof.
\end{proof}
Using the constant $C_1$ from Lemma \ref{lemma_lower_bound_b_eps} we define the cut-off operator
\begin{align}\label{cut-off_operator_Pi_eps}
\Pi_\eps: \mathsf{L}^2\at{\mathbb{R}}\to \mathsf{L}^2\at{\mathbb{R}}
\end{align}
by its symbol function
\begin{align}\notag
\pi_\eps\at{k}:=\begin{cases}
1                      &\text{ for } |k|\leq \frac{C_1}{\eps}\,,\\
0                      &\text{ for } |k|>\frac{C_1}{\eps}
\end{cases}
\end{align}
and derive the following result.
\begin{proposition}[$\calB_\eps^{-1}$ and cut-off in Fourier space]\label{proposition_cut_off_B_eps^-1}
For any $0<\eps\leq 1$ and all $G\in\mathsf{L}^2\at{\mathbb{R}}$ we have
\begin{align}\notag
\left\| \Pi_\eps\,\calB_\eps^{-1} G\right\|_{2,2}
+\frac{1}{\eps^2}\left\|\left(\operatorname{Id}-\Pi_\eps\right)\calB_\eps^{-1} G\right\|_2\leq D\|G\|_2\,,
\end{align}
for some constant $D>0$ independent of $\eps$, where $\|\cdot\|_{2,2}$ denotes the norm in $\mathsf{W}^{2,2}\at{\mathbb{R}}$.
\end{proposition}
\begin{proof}
Using Lemma \ref{lemma_lower_bound_b_eps}, Fourier transform, and Parseval's theorem we obtain
\begin{align*}
\left\| \Pi_\eps\,\calB_\eps^{-1} G\right\|_{2,2}^2
&=\frac{1}{2\,\pi} \int\limits_{|k|\leq \frac{C_1}{\eps}} \frac{\left(1+k^2+k^4\right)|\widehat{G}\at{k}|^2}{b_\eps^2\at{k}}\dint{k}
\leq \frac{1}{2\,\pi} \int\limits_{|k|\leq \frac{C_1}{\eps}} \frac{\left(1+k^2+k^4\right)|\widehat{G}\at{k}|^2}{C_0^2\left(1+k^2\right)^2}\dint{k}
\leq\frac{\|G\|_2^2}{C_0^2}
\end{align*}
as well as
\begin{align*}
\left\|\left(\operatorname{Id}-\Pi_\eps\right)\calB_\eps^{-1} G\right\|_2^2
&=\frac{1}{2\,\pi} \int_{|k|> \frac{C_1}{\eps}} \,\frac{|\widehat{G}\at{k}|^2}{b_\eps^2\at{k}}\dint{k}
\leq \eps^4\frac{\|G\|_2^2}{C_0^2}\,.
\end{align*}
The desired estimate now follows with $D:=2/C_0$.
\end{proof}
Although the operator $\Pi_\eps\,\calB_\eps^{-1}$ is much more regular than $\calB_\eps^{-1}$, it is not yet compact. In the proof of Proposition \ref{proposition_uniform_invertibility} we therefore  introduce an additional cut-off in position space.
%
%
\subsection{Properties of the operators \texorpdfstring{$\calQ_\eps$}{Q} and \texorpdfstring{$\calP_\eps$}{P}}\label{subsection_operators_properties_2}
%
%
In order to show that the operator equation (\ref{operator_equation_peridynamics}) transforms for $\eps\to0$ into the ODE (\ref{KdV_equation_peridynamic}), we must characterize the limiting behaviour of $\calQ_\eps$ and define
\begin{align}\notag
\calQ_0[W]:=\left(\int_0^\infty \beta\at{\xi}\,\xi^3\dint{\xi}\right)W^2\,.
\end{align}
\begin{lemma}[convergence of $\calQ_\eps$]\label{lemma_convergence_Q_eps}
We have 
\begin{align*}
\|\calQ_\eps [W]- \calQ_0 [W]\|_2 \xrightarrow{\;\eps \to 0\;} 0
\end{align*}
for all $W\in \mathsf{L}^2\at{\mathbb{R}}\cap \mathsf{L}^\infty\at{\mathbb{R}}$ and
\begin{align*}
\|\calQ_\eps[W]-\calQ_0[W]\|_2\leq C\,\eps^2
\end{align*} 
for all $W\in\mathsf{W}^{2,2}\at{\mathbb{R}}$.
\end{lemma}
\begin{proof}
Let $W\in\mathsf{L}^2\at{\mathbb{R}}\cap \mathsf{L}^\infty\at{\mathbb{R}}$ be fixed. By Proposition \ref{prop:Bochner1}, the function $F_\eps: \pair{0}{\infty}\to\mathsf{L}^2\at{\mathbb{R}}$ with
\begin{align*}
F_\eps\at{\xi}:= \beta\at{\xi}\,\xi^3 \calA_{\eps\xi}\left(\calA_{\eps\xi} W\right)^2
\end{align*}
is Bochner integrable with respect to $\xi>0$ and the convergence result (\ref{convergence_A_etaW_to_W}) implies that
\begin{align*}
F_\eps\at{\xi}\xrightarrow{\;\eps\to 0\;} F_0\at{\xi}:=\beta\at{\xi}\,\xi^3\, W^2
\qquad
\text{ strongly in }
\;
 \mathsf{L}^2\at{\mathbb{R}}
\end{align*}
holds  pointwise in $\xi$. Moreover, from (\ref{estimate_A_eta_1})$_2$ with $p=2$ we deduce the estimate
\begin{align*}
\|F_\eps\at{\xi}\|_2 
\leq  \beta\at{\xi}\,\xi^3\,\|W\|_\infty\, \|W\|_2
=:g\at{\xi}\,,
\end{align*}
where $g:\pair{0}{\infty}\to \mathbb{R}$ is integrable according to (\ref{integral_beta})$_1$. The Dominated Convergence Theorem for Bochner integrals (see Theorem \ref{theorem_lebesgue_bochner}) guarantees that $F_0:\pair{0}{\infty}\to\mathsf{L}^2\at{\mathbb{R}}$ is Bochner integrable and the first assertion follows via
\begin{align*}
\|\calQ_\eps [W]-\calQ_0 [W]\|_2 \leq \int_0^\infty \|F_\eps\at{\xi}-F_0\at{\xi}\|_2 \dint{\xi}\xrightarrow{\;\eps\to 0\;} 0
\end{align*}
from Theorem \ref{theorem_bochner}. We now assume that $W$ belongs to $\mathsf{W}^{2,2}\at{\mathbb{R}}$. By (\ref{error_A_eta_1})$_1$ we have
\begin{align*}
\|\calA_{\eps\xi} W-W\|_2
\leq C\,\eps^2\,\xi^2\,\|W''\|_2
\leq C\,\eps^2\,\xi^2
\end{align*}
for all $\xi>0$ and the estimate
\begin{align*}
\|\calA_{\eps\xi} W^2-W^2\|_2
\leq C\,\eps^2\,\xi^2\, \|(W^2)''\|_2
\leq C\,\eps^2\,\xi^2
\end{align*}
can be justified using standard embedding theorems. In combination with (\ref{estimate_A_eta_1})$_2$ we get
\begin{align*}
\|\calA_{\eps\xi}\left(\calA_{\eps\xi} W\right)^2-W^2\|_2
&\leq \|\calA_{\eps\xi}\left(\calA_{\eps\xi} W\right)^2-\calA_{\eps\xi} W^2\|_2+\|\calA_{\eps\xi} W^2-W^2\|_2\\
&\leq \left( \|\calA_{\eps\xi}W\|_\infty +\|W\|_\infty \right)\|A_{\eps\xi} W-W\|_2+C\,\eps^2\,\xi^2\\
&\leq C\,\eps^2\,\xi^2
\end{align*}
and hence
\begin{align*}
\|\calQ_\eps[W]-\calQ_0[W]\|_2
\leq C \,\eps^2 \int_0^\infty \beta\at{\xi}\,\xi^5\dint{\xi}
\leq C\,\eps^2
\end{align*}
due to (\ref{integral_beta})$_3$.
\end{proof}
Passing to the formal limit $\eps\to 0$ in (\ref{operator_equation_peridynamics}) yields
\begin{align}
\label{operator_equation_B_0_Q_0_peridynamic}
\calB_0 W=\calQ_0[W]
\end{align}
since $\eps^2\,\calP_\eps[V]\xrightarrow{\;\eps\to 0\;} 0$  holds for all $W\in\mathsf{L}^2\at{\mathbb{R}}$ according to Proposition \ref{prop:Bochner1}. This equation is the operator representation of the KdV traveling wave equation (\ref{KdV_equation_peridynamic}) with parameters
\begin{align}
\label{KdV_equation_coeffizients}
d_1:=\frac{12}{\int_0^\infty\alpha\at{\xi}\,\xi^4\dint{\xi}}\,,\qquad\qquad
d_2:=\frac{12\int_0^\infty \beta\at{\xi}\,\xi^3\dint{\xi}}{\int_0^\infty \alpha\at{\xi}\, \xi^4\dint{\xi}}\,,
\end{align} 
and the smooth function $W_0$ from \eqref{Kdv_equation_solution} is the unique solution in the space $\fspaceL^2_\even\at\Rset$.
%
%
\section{Proof of the main result}\label{section_main_result}
%
%
In this section we prove our main result from \S\ref{sect:intro} and show that there exists a locally unique solution $W_\eps\in\mathsf{L}_{\mathrm{even}}^2\at{\mathbb{R}}$ to (\ref{operator_equation_peridynamics}) that lies in a small neighborhood of the KdV solution $W_0$. To do this, we use the predictor-corrector approach (\ref{predictor_corrector}) and thus exclude the trivial solution $W_\eps\equiv 0$ from our considerations. First we show that $W_0$ is an approximate solution to the $\eps$-problem (\ref{operator_equation_peridynamics}).
\begin{lemma}[consistency]\label{lemma_consistency}
There exists a constant $C>0$ such that
\begin{align}\label{residual_peridynamic}
\|K_\eps\|_2+\|E_\eps\|_2\leq C
\qquad
\text{ with } 
\qquad
K_\eps:= \frac{\calQ_\eps[W_0]-\calB_\eps[W_0]}{\eps^2}\,,
\qquad
E_\eps:=\calP_\eps[W_0]
\end{align} 
holds for any  $0<\eps\leq 1$, where $K_\eps+E_\eps$ is the $\eps$-residual of $W_0$.
\end{lemma}
\begin{proof}
By applying Lemma \ref{lemma_convergence_B_eps_to_B_0} and Lemma \ref{lemma_convergence_Q_eps} to $W_0\in\mathsf{C}^\infty\at{\mathbb{R}}$ we obtain
\begin{align*}
\|K_\eps\|_2
\leq \frac{\|\calQ_\eps[W_0]-\calQ_0[W_0]\|_2}{\eps^2}+\frac{\|\calB_\eps W_0-\calB_0 W_0\|_2}{\eps^2}
\leq C
\end{align*}
since $W_0$ is a solution to \eqref{operator_equation_B_0_Q_0_peridynamic}. Due to (\ref{estimate_A_eta_1})$_2$ we further have 
\begin{align*}
\left\| (\eps^2\xi\,\calA_{\eps\xi} W_0)^3\right\|_2
&\leq \eps^6\xi^3\,\|\calA_{\eps\xi} W_0\|_\infty^2\,\|\calA_{\eps\xi} W_0\|_2
\leq \eps^6\xi^3\,\| W_0\|_\infty^2\,\|W_0\|_2
\end{align*}
for all $\xi\in\pair{0}{\infty}$, so 
\begin{align*}
\|E_\eps\|_2
&\leq \|W_0\|_\infty^2\,\|W_0\|_2\int_0^\infty \gamma\at{\xi}\,\xi^4\dint{\xi}
\leq C\,,
\end{align*} 
is provided by (\ref{integral_gamma})$_2$ and the theory of Bochner integrals, see the proof of Proposition \ref{prop:Bochner1}.
\end{proof}
%
%
\subsection{The operators \texorpdfstring{$\calL_\eps$}{L} and \texorpdfstring{$\calL_0$}{L0}}\label{subsection_operator_L_eps}
%
%
We substitute the ansatz (\ref{predictor_corrector}) into (\ref{operator_equation_peridynamics}) and use both the linearity of $\calB_\eps$ and the quadraticity of $\calQ_\eps$. After rearranging terms and dividing by $\eps^2$ wo obtain the equation
\begin{align}\notag
\calB_\eps V_\eps -\calM_\eps V_\eps=\frac{\calQ_\eps [W_0]-\calB_\eps W_0}{\eps^2}+\eps^2\,\calQ_\eps [V_\eps] +\calP_\eps [W_0+\eps^2\,V_\eps]\,,
\end{align}
where the operator $\calM_\eps:\mathsf{L}^2\at{\mathbb{R}}\to \mathsf{L}^2\at{\mathbb{R}}$ with
\begin{align}\label{operator_M_eps_peridynamic}
\calM_\eps V:= 2\int_0^\infty \beta\at{\xi}\,\xi^3 \calA_{\eps\xi}\left( \left(\calA_{\eps\xi} W_0\right)\left(\calA_{\eps\xi} V\right)\right)\dint{\xi}
\end{align} 
is well-defined in the sense of Bochner integrals, see Proposition \ref{prop:Bochner1}. This equation can also be written as
\begin{align}\label{operatorequation_L_eps_peridynamic}
\calL_\eps V_\eps =K_\eps + E_\eps +\eps^2\,\calQ_\eps [V_\eps]+\eps^2\, \calN_\eps [V_\eps]
\end{align} 
with 
\begin{align}\label{operator_L_eps_peridynamic}
\calL_\eps V :=\calB_\eps V-\calM_\eps V\,,\qquad \calN_\eps [V]:=\frac{\calP_\eps[W_0+\eps^2\, V_\eps]-\calP_\eps [W_0]}{\eps^2}
\end{align}
and residual $K_\eps+E_\eps$ as in (\ref{residual_peridynamic}). Passing to the limit $\eps\to0$ yields the formal limit operators $\calL_0 : \mathsf{W}^{2,2}\at{\mathbb{R}}\to \mathsf{L}^2\at{\mathbb{R}}$  and $\calM_0 : \mathsf{L}^{2}\at{\mathbb{R}}\to \mathsf{L}^2\at{\mathbb{R}}$ with
\begin{align}\notag
\calL_0 V:=\calB_0 V-\calM_0 V\,, 
\qquad
\calM_0 V:= 2\left( \int_0^\infty \beta\at{\xi}\,\xi^3\dint{\xi}\right) W_0\, V\,.
\end{align}
\begin{lemma}[elementary properties of $\calL_\eps$]\label{lemma_properties_of_L_eps}
For any $\eps>0$, the operator $\calL_\eps$ is self-adjoint in the $\mathsf{L}^2$-sense and admits the invariant space $ \mathsf{L}^{2}_\even\at{\mathbb{R}}$. Moreover, we have
\begin{align*}
\calL_\eps V \xrightarrow{\eps \to 0} \calL_0 V 
\qquad
\text{ strongly in }
\quad \mathsf{L}^2\at{\mathbb{R}}
\end{align*}
for any $V\in \mathsf{W}^{2,2}\at{\mathbb{R}}$. 
\end{lemma}
\begin{proof}
The self-adjointness can be checked directly and the invariance of even functions follows from the properties of $\calB_\eps$, $\calA_{\eps \xi}$ and $W_0$. Lemma \ref{lemma_convergence_B_eps_to_B_0} implies that $\calB_\eps V$ converges strongly in $\mathsf{L}^2\at{\mathbb{R}}$ to $\calB_0 V$ for all $V\in\mathsf{W}^{2,2}\at{\mathbb{R}}$ and the strong convergence of $\calM_\eps V$ to $\calM_0 V$ follows analogously to the proof of Lemma \ref{lemma_convergence_Q_eps} from Theorem \ref{theorem_lebesgue_bochner}.
\end{proof}
The properties of the nonautonomous differential operator $\calL_0$ are well understood and imply that the nonlocal operators $\calL_\eps$ cannot be uniformly invertible on $\fspaceL^2\at{\mathbb{R}}$. In the next subsection we therefore restrict our considerations to even functions $V_\eps$.
\begin{lemma}[elementary properties of $\calL_0$]\label{lemma_properties_of_L_0}
The operator $\calL_0$ is  self-adjoint in the $\mathsf{L}^2$-sense. Moreover its  $\mathsf{W}^{2,2}\at{\mathbb{R}}$-kernel is one dimensional and given by the span of the odd function $W_0^\prime$.  
\end{lemma}
\begin{proof}
The first part follows immediately and the second part is a consequence of classical Sturm-Liouville arguments, see \cite[Lemma 3.1]{HML16} for the details.
\end{proof}
%
%
\subsection{Uniform invertibility of  the operator \texorpdfstring{$\calL_\eps$}{L}}\label{subsection_invertibility_L_eps}
%
%
In this section we establish uniform invertibility estimates for $\calL_\eps$ in the space $\fspaceL^2_\even\at{\mathbb{R}}$. This result forms the core of our asymptotic analysis as it transforms \eqref{operatorequation_L_eps_peridynamic} into a fixed-point problem that can be tackled by the Contraction Mapping Principle.
\begin{proposition}[auxiliary result for the uniform invertibility of $\calL_\eps$]\label{proposition_uniform_invertibility}
Let $\at{\eps_n}_{n\in\mathbb{N}}$ be given with $\eps_n \to0$. Then, the implication
\begin{align*}
\| \calL_{\eps_n} {V_n}\| _2\;\; \xrightarrow{\;n\to\infty\;}\;\; 0\qquad \implies \qquad 
\|  {V_n}\| _2\;\; \xrightarrow{\;n\to\infty\;}\;\; 0
\end{align*}
holds for any bounded sequence $\at{V_n}_{n\in\mathbb{N}}$ in $\mathsf{L}^2_{\mathrm{even}}\at{\mathbb{R}}$.
\end{proposition}
\begin{proof} 
To proof the assertion by contradiction we consider a not relabeled subsequence that satisfies
\begin{align}
\label{lemma_uniform_invertibility-e1} \lim_{n\to\infty} \| \calL_{\eps_n} {V_n}\| _2=0
\end{align}
as well as
\begin{align}
\label{lemma_uniform_invertibility-e2}
0<\eps_n\leq 1\,,\qquad \delta\leq \|{V_n}\| _2\leq 1
\end{align}
for some fixed $\delta>0$ and all $n\in\Nset$.
\par\underline{\textit{Weak convergence to 0:}} %
Since $\mathsf{L}_{\mathrm{even}}^2\at{\mathbb{R}}$ is a closed subspace of $\mathsf{L}^2\at{\mathbb{R}}$ we find $V_\infty\in \mathsf{L}_{\mathrm{even}}^2\at{\mathbb{R}}$ such that
\begin{align}\label{invertibility_L_eps_weak_convergence_V_n_j}
V_{n}\xrightharpoonup{n\to\infty} V_\infty 
\quad
\text{ weakly in } 
\quad \mathsf{L}_{\mathrm{even}}^2\at{\mathbb{R}}
\end{align} 
holds along a not relabeled subsequence and in view of Lemma \ref{lemma_properties_of_L_eps} we have
\begin{align*}
\langle V_\infty, \calL_0\,\varphi \rangle_2 =\lim_{n\to\infty} \langle V_{n}, \calL_{\eps_{n}}\varphi \rangle_2=\lim_{n\to\infty} \langle \calL_{\eps_{n}}V_{n}, \varphi \rangle_2 =0
\end{align*}
for any sufficiently smooth test function $\varphi$, where $\langle \cdot,\cdot\rangle_2$ denotes the standard scalar product. Using the definition of $\calL_0$ and rearranging terms gives
\begin{align*}
\left| \int_\mathbb{R} V_\infty \at{x}\, \varphi''\at{x}\dint{x} \right|
\leq C\, \|\varphi\|_2
\end{align*}
for all $\varphi\in C_{\mathrm{c}}^\infty\at{\mathbb{R}}$ and standard arguments --- see for instance \cite[Proposition 8.3]{brezis} --- imply \mbox{$V_\infty\in\mathsf{W}^{2,2}\at{\mathbb{R}}$} and hence $\langle\calL_0 V_\infty, \varphi \rangle_2=\langle V_\infty, \calL_0\,\varphi \rangle_2 =0$. In particular, the even function $V_\infty$ belongs to the kernel of $\calL_0$ and vanishes acccording to Lemma \ref{lemma_properties_of_L_0}.
\par\underline{\textit{Further notations and preliminary estimates:}} %
We choose $H\in\pair{0}{\infty}$ sufficiently large such that 
\begin{align}\label{invertibility_L_eps_constant_L_delta_1}
 \int_H^\infty \beta\at{\xi}\,\xi^3\dint{\xi}
\leq \frac{\delta}{8\,D\,\|W_0\|_\infty }
\end{align}
holds with constant $D$ from Proposition \ref{proposition_cut_off_B_eps^-1} and afterwards we  choose $L>H$ sufficiently large such that 
\begin{align}\label{invertibility_L_eps_estimate_W_0}
\sup_{|y|\geq L-H} W_0\at{y} \leq \frac{\delta}{8\,D\int_0^H \beta\at{\xi}\,\xi^3\dint{\xi}}\,.
\end{align} 
We split $V_{n}=V_{n}^{(1)}+V_{n}^{(2)}+V_{n}^{(3)}$ according to
\begin{align*}
V_{n}^{(1)}:=\left(\operatorname{Id}-\Pi_{\eps_{n}}\right)V_{n}\,, 
\qquad
V_{n}^{(2)}:=\chi_{L}\,\Pi_{\eps_{n}}\,V_{n}\,,
\qquad
V_{n}^{(3)}:=\left(1-\chi_{L}\right)\Pi_{\eps_{n}}\,V_{n}
\end{align*}
into three parts, where $\chi_{L}$ denotes the characteristic function of the  interval $I_{L}:=[-L,+L]$ and $\Pi_{\eps_n}$ is the cut-off operator from (\ref{cut-off_operator_Pi_eps}). Our definitions, the properties of orthogonal projections, and estimate \eqref{lemma_uniform_invertibility-e2} imply 
\begin{align}\label{invertibility_L_eps_estimate_V_n_j}
\max_{i\in \{1,2,3\}} \|V_{n}^{(i)}\|_2\leq \|V_{n}\|_2\leq 1
\end{align}
for all $n\in\Nset$. Setting $U_{n}^{(i)}:=\calM_{\eps_{n}} V_{n}^{(i)}$ we further obtain
\begin{align}
\label{invertibility_L_eps_estimate_U_n_j}
\begin{split}
\|U_{n}^{(i)}\|_2
& \leq 2 \int_0^\infty \beta\at{\xi}\,\xi^3 \, \|\bat{\calA_{\eps_n\xi}W_0}\,\bat{\calA_{\eps_n\xi} V_{n}^{(i)}}\|_2 \dint{\xi}\\
&\leq 2\, \|W_0\|_\infty \, \|V_{n}^{(i)}\|_2 \int_0^\infty \beta\at{\xi}\,\xi^3\dint{\xi}\\
&\leq C\,\|V_{n}^{(i)}\|_2\,,
\end{split}
\end{align}
where we used (\ref{integral_beta})$_1$, (\ref{estimate_A_eta_1})$_2$ as well as the properties of Bochner integrals (see the appendix). Moreover, 
our definitions combined with Proposition \ref{proposition_cut_off_B_eps^-1} (applied to $G=\calM_{\eps_{n}}V_{n}+\calL_{\eps_{n}} V_{n}$) provide 
\begin{align}
\label{invertibility_L_eps_essential_estimate}
\|V_{n}^{(2)}+V_{n}^{(3)}\|_{2,2}+\eps_{n}^{-2}\|V_{n}^{(1)}\|_2
&\leq D\,\at{\|U_{n}^{(1)}\|_2+\|U_{n}^{(2)}\|_2+\|U_{n}^{(3)}\|_2+\|\calL_{\eps_n} V_{n}\|_2}
\end{align}
thanks to $V_{n}=\calB_{\eps_n}^{-1}\left(\calM_{\eps_{n}} V_{n}+\calL_{\eps_n} V_{n}\right)$ and $\calM_{\eps_{n}} V_{n}=\sum_{i=1}^3U_{n}^{(i)}$.
\par\underline{\textit{Strong convergence of $V_{n}^{(1)}$ and $V_{n}^{(2)}$:}} %
The estimates \eqref{invertibility_L_eps_estimate_V_n_j}, \eqref{invertibility_L_eps_estimate_U_n_j}, and  \eqref{invertibility_L_eps_essential_estimate} imply  $\|V_{n}^{(1)}\|_2\leq C\,\eps_{n}^2$  and hence
\begin{align}\label{invertibility_L_eps_strong_convergence_V_n_j^1}
V_{n}^{(1)}\xrightarrow{\;n\to\infty\;} 0
\quad
\text{ strongly in } 
\quad
 \mathsf{L}_{\mathrm{even}}^2\at{\mathbb{R}}\,.
\end{align} 
Since $V_{n}^{(3)}$ vanishes inside the interval $I_L$ by construction, we also infer from these estimates and \eqref{lemma_uniform_invertibility-e1} the uniform bound 
\begin{align}\notag
\|V_{n}^{(2)}\|_{2,2,I_{L}} 
\leq \| V_{n}^{(2)}+V_{n}^{(3)}\|_{2,2}
\leq C\,D\,,
\end{align} 
where $\|\cdot\|_{2,2,I_{L}}$ denotes the norm in $\mathsf{W}^{2,2}\at{I_{L}}$. The Rellich-Kondrachov Theorem states that $\mathsf{W}^{2,2}\at{I_{L}}$ is compactly embedded into $ \mathsf{L}^2\at{I_{L}}$, so there exist subsequences such that $V_{n}^{(2)}$ converges strongly in $\fspaceL^2\at{I_L}$. However, any such limit function must vanish due to \eqref{invertibility_L_eps_weak_convergence_V_n_j}, $V_\infty=0$,  \eqref{invertibility_L_eps_strong_convergence_V_n_j^1}, and the support properties of $V_{n}^{(3)}$. The combination of compactness and uniqueness of strong accumulation points on the interval $I_L$ finally implies
\begin{align}\label{invertibility_L_eps_strong_convergence_V_n_j^2}
V_{n}^{(2)}\xrightarrow{\;n\to\infty\;} 0
\quad
\text{ strongly in } 
\quad \mathsf{L}_{\mathrm{even}}^2\at{\mathbb{R}}
\end{align}
since $V_{n}^{(2)}$ vanishes outside the interval $I_L$.
\par
\underline{\textit{Upper bounds for $\|U_{n}^{(3)}\|_2$:}} %
We split the $\xi$-integral in the formula for $\calM_{\eps_n}V_{n}^{(3)}$, see formula \eqref{operator_M_eps_peridynamic}, into two parts corresponding to $\xi\in [0,\, H]$ and $\xi\in (H,\, \infty)$. Concerning the first term we observe
\begin{align*}
\left|\left( (\calA_{\eps_n \xi}W_0)\at{x} \right)\bat{ (\calA_{\eps_n \xi} V_{n}^{(3)})\at{x} }\right|
&\leq \left(\sup_{|x|>L-H/2}\left| (\calA_{\eps_n \xi}W_0)\at{x}\right| \right) (\calA_{\eps_n \xi} |V_{n}^{(3)}|)\at{x}\\
&\leq \left(\sup_{|y|>L-H} W_0\at{y}\right) (\calA_{\eps_n \xi} |V_{n}^{(3)}|)\at{x}\,,
\end{align*}
where we used  (\ref{estimate_A_eta_pointwise}), $\eps_n\,\xi\leq H$ and that $\calA_{\eps_n \xi} V_{n}^{(3)}$ is supportet in the set \mbox{$\left\{x\in\mathbb{R}\,:\, |x|>L-\eps_n\, \xi/2\right\}$}. By (\ref{estimate_A_eta_1})$_2$ we therefore get 
\begin{align*}
\left\| \calA_{\eps_n \xi}\left( (\calA_{\eps_n \xi}W_0) (\calA_{\eps_n \xi} V_{n}^{(3)})\right)\right\|_2
\leq \left(\sup_{|y|>L-H} W_0\at{y}\right) \| V_{n}^{(3)}\|_2
\end{align*}
and control the first integral contribution to $U_{n}^{(3)}$  by 
\begin{align}\notag
\begin{split}
\left\| 2\int_0^H \beta\at{\xi}\,\xi^3 \, \calA_{\eps_n\, \xi}\bat{ \bat{\calA_{\eps_n\xi}W_0}\,\bat{\calA_{\eps_n\xi} V_{n}^{(3)}}} \dint{\xi}\right\|_2
&\leq 2\left(\sup_{|y|>L-H} W_0\at{y}\right) \| V_{n}^{(3)}\|_2 \int _0^H \beta\at{\xi}\,\xi^3 \dint{\xi}\\
&\leq \frac{\delta}{4\,D}
\end{split}
\end{align}
thanks to (\ref{invertibility_L_eps_estimate_V_n_j}) and the choice of $L$ in \eqref{invertibility_L_eps_estimate_W_0}. The second contribution can be estimated by 
\begin{align}\notag
\begin{split}
\left\| 2\int_H^\infty \beta\at{\xi}\,\xi^3 \, \calA_{\eps_n\, \xi}\bat{ \bat{\calA_{\eps_n\xi}W_0}\,\bat{\calA_{\eps_n\xi} V_{n}^{(3)}}} \dint{\xi}\right\|_2
&\leq 2 \, \|W_0\|_\infty \,\int _H^\infty \beta\at{\xi}\,\xi^3 \dint{\xi}
\leq \frac{\delta}{4\,D}
\end{split}
\end{align}
due to the choice of $H$ in \eqref{invertibility_L_eps_constant_L_delta_1}. Together we obtain 
\begin{align}\label{invertibility_L_eps_upper_bound_U_n_j}
\|U_{n}^{(3)}\|_2
\leq\frac{\delta}{2\,D}\,.
\end{align}
\underline{\textit{Derivation of the contradiction:}} 
Combining (\ref{invertibility_L_eps_estimate_U_n_j}) with (\ref{invertibility_L_eps_essential_estimate}) leads to
\begin{align*}
\|V_{n}\|_2
&\leq D\left( C\, \|V_{n}^{(1)}\|_2 + C\, \|V_{n}^{(2)}\|_2+ \|U_{n}^{(3)}\|_2+\|\calL_{\eps_n}V_{n}\|_2\right) 
\end{align*}
and (\ref{lemma_uniform_invertibility-e1}), (\ref{invertibility_L_eps_strong_convergence_V_n_j^1}), (\ref{invertibility_L_eps_strong_convergence_V_n_j^2}) and (\ref{invertibility_L_eps_upper_bound_U_n_j}) imply
\begin{align*}
\limsup_{n\to\infty} \|V_{n}\|_2\leq D\, \limsup_{n\to\infty} \|U_{n}^{(3)}\|_2\leq \frac{\delta}{2}\,.
\end{align*}
This, however, contradicts \eqref{lemma_uniform_invertibility-e2} and the proof is complete. 
\end{proof}
\begin{cor}[uniform invertibility of $\calL_\eps$]\label{corollary_uniform_invertibility_L_eps}
The operator $\calL_\eps$ is uniformly invertible for all small $\eps>0$. More precisely, for any sufficiently small $ \eps_*>0$ there exists a constant $C$ (which may depend on $\eps_*$ but not on $\eps$) such that
\begin{align*}
\|\calL_\eps^{-1} G\|_2\leq C\,\|G\|_2
\end{align*} 
holds for all $0<\eps\leq \eps_*$ and any $G\in\mathsf{L}^2_{\mathrm{even}}\at{\mathbb{R}}$.
\end{cor}
\begin{proof}
As a direct consequence of Proposition \ref{proposition_uniform_invertibility} we obtain
the existence of a constant $C_*>0$ such that
\begin{align}\notag
\|\calL_\eps V\|_2\geq C_*\,\|V\|_2
\end{align}
for all $V\in\mathsf{L}_{\mathrm{even}}^2\at{\mathbb{R}}$ and $0<\eps\leq \eps_*$. Using this as well as the formula $\ker\calL_\eps=\operatorname{coker}\calL_\eps$  (which holds for any continuous and self-adjoint linear operator) we conclude that $\calL_\eps$ is both injective and surjective, and that $C=1/C_*$ can be used as continuity constant of the inverse.
\end{proof}
%
%
%
\subsection{The nonlinear fixed point argument}\label{subsection_fixed_point_argument}
%
%
In view of the previous result and based on (\ref{operatorequation_L_eps_peridynamic}) we define the operator $\calF_\eps : \mathsf{L}_{\mathrm{even}}^2\at{\mathbb{R}}\to\mathsf{L}_{\mathrm{even}}^2\at{\mathbb{R}}$ by
\begin{align}\label{operator_F_eps_peridynamic}
\calF_\eps[V] :=\calL_\eps ^{-1}\left( K_\eps + E_\eps +\eps^2\,\calQ_\eps [V]+\eps^2\, \calN_\eps [V]\right)\,.
\end{align}
\begin{theorem}[nonlinear fixed point argument]
There exists $0<\eps_*\leq 1$ such that the operator $\calF_\eps$ admits for any $0<\eps\leq \eps_*$ a unique fixed point $V_\eps$ in the set $B_D=\{V\in \mathsf{L}_{\mathrm{even}}^2\at{\mathbb{R}}\,:\, \|V\|_2\leq D\}$, where $D>0$ is a sufficiently large  constant that  may depend on $\eps_*$ but not on $\eps$.
\end{theorem}
\begin{proof}
We show that the operator $\calF_\eps: B_D\to B_D$ is a contractive self-mapping for sufficiently large $D$; the claim is then a direct consequence of the Contraction Mapping Principle. We identify the value of $D$ at the end of this proof and denote by $C>0$ any generic constant that is independent of both $D$ and $\eps$.
\par\underline{\textit{Estimates for the quadratic terms:}} 
For $V_1, V_2\in B_D$ we deduce
\begin{align*}
|\eps^2\,\calQ_\eps [V_2]-\eps^2\,\calQ_\eps [V_1]|
&\leq \eps^2 \int_0^\infty \beta\at{\xi}\,\xi^3 \left(\|\calA_{\eps\xi}V_2\|_\infty+\|\calA_{\eps\xi}V_1\|_\infty\right)\calA_{\eps\xi}^2|V_2-V_1|\dint{\xi}\\
&\leq  \eps^{3/2}\, C\,D \int_0^\infty \beta\at{\xi}\,\xi^{5/2}\calA_{\eps\xi}^2\left|V_2-V_1\right|\dint{\xi}\,,
\end{align*}
from (\ref{estimate_A_eta_1})$_1$, (\ref{estimate_A_eta_pointwise}) and in view of (\ref{estimate_A_eta_1})$_2$,  (\ref{integral_beta})$_2$  we obtain
\begin{align*}
 \|\eps^2\,\calQ_\eps [V_2]-\eps^2\,\calQ_\eps [V_1]\|_2
 &\leq  \eps^{3/2}\, C\,D \int_0^\infty \beta\at{\xi}\,\xi^{5/2}\|\calA_{\eps\xi}^2\left|V_2-V_1\right|\|_2\dint{\xi}\\
 &\leq  \eps^{3/2}\, C\,D\|V_2-V_1\|_2\,.
\end{align*}
Moreover, the choice $V_2:=V$ and $V_1:=0$ implies
\begin{align*}
\|\eps^2\,\calQ_\eps [V]\|_2
\leq  \eps^{3/2}\, C\,D\|V\|_2 
\leq  \eps^{3/2}\, C\,D^2
\end{align*}
for any $V\in B_D$.
\par\underline{\textit{Estimates for the higher order terms:}} %
For $V_1, V_2 \in B_D$  we define $Z_{\xi,\eps,i}:=\eps^2\,\xi\, \calA_{\eps\xi}(W_0+\eps^2\,V_i)$ and derive
 $|Z_{\xi,\eps,2}-Z_{\xi,\eps,1}|=\eps^4\,\xi |\calA_{\eps\xi}(V_2-V_1)|$ from \eqref{estimate_A_eta_pointwise} as well as 
\begin{align*}
\begin{split}
\|Z_{\xi,\eps,i}\|_\infty 
&\leq \begin{cases}
\eps^2\,\xi^{1/2} \bat{ C+\eps^{3/2}\,D} & \text{ for } 0\leq\xi\leq 1\\
\eps^2\,\xi \bat{ C+\eps^{3/2}\,D} & \text{ for } 1<\xi<\infty
\end{cases}
\end{split}
\end{align*}
from (\ref{estimate_A_eta_1})$_1$ and (\ref{estimate_A_eta_1})$_2$. Combining this with (\ref{estimate_A_eta_pointwise}), (\ref{estimate_second_derivative_psi}) and using the meanvalue theorem we obtain the pointwise estimate
\begin{align*}
\big|\eps^2\, \calN_\eps [V_2]-\eps^2\, \calN_\eps [V_1]\big|
&\leq \frac{1}{\eps^6} \int_0^\infty \xi \calA_{\eps\xi}\left|\partial_r\psi\pair{Z_{\xi,\eps,2}}{\xi}-\partial_r\psi\pair{Z_{\xi,\eps,1}}{\xi}\right| \dint{\xi}\\
&\leq \eps^2\, \bat{C+\eps^{3/2}\,D}^2 \int_0^1 \xi^3\,\gamma\at{\xi}\, \calA_{\eps\xi}^2\left|V_2-V_1\right| \dint{\xi}\\
&+\eps^2\, \bat{C+\eps^{3/2}\,D}^2 \int_1^\infty \xi^4\,\gamma\at{\xi}\, \calA_{\eps\xi}^2\left|V_2-V_1\right| \dint{\xi}
\end{align*}
and hence the Lipschitz bound
\begin{align*}
\|\eps^2\,\calN_\eps[V_2]-\eps^2\,\calN_\eps[V_1]\|_2
&\leq  \eps^2\,C\, \bat{C+\eps^{3/2}\,D}^2 \|V_2-V_1\|_2
\end{align*}
thanks to (\ref{estimate_A_eta_1})$_2$ and  (\ref{integral_gamma}). Moreover, setting $V_2:=V$ and $V:=0$   we get
\begin{align*}
\|\eps^2\,\calN_\eps[V]\|_2
\leq \eps^2\,C\, \bat{C+\eps^{3/2}\,D}^2 \|V\|_2
\leq \eps^2\,C\,D\, \bat{C+\eps^{3/2}\,D}^2 
\end{align*}
for all $V\in B_D$.
\par\underline{\textit{Concluding arguments:}} %
Combining all partial results derived so far with Lemma \ref{lemma_consistency} and Corollary \ref{corollary_uniform_invertibility_L_eps} gives
\begin{align*}
\|\calF_\eps[V]\|_2
\leq C+ \eps^{3/2}\,C\,D^2 +\eps^2\,C\,D \bat{C+\eps^{3/2}\,D}^2
\end{align*}
for all $V\in B_D$ as well as
\begin{align*}
\|\calF_\eps[V_2]-\calF_\eps[V_1]\|_2
\leq \left(\eps^{3/2}\,C\,D +\eps^2\,C\, \bat{C+\eps^{3/2}\,D}^2\right)\|V_2-V_1\|_2
\end{align*}
for all $V_1, V_2 \in B_D$. Setting $D:=2\,C$ and choosing $\eps$ sufficently small we finally obtain 
\begin{align*}
F_\eps[V]\|\leq D\,,\qquad \|F_\eps[V_2]-F_\eps[V_1]\|_2\leq \kappa \|V_2-V_1\|_2\,,
\end{align*}
where the contraction number $\kappa\in(0,1)$ can be made arbitrarily small.
\end{proof}
%
%
%
\appendix
%
%
\section{Bochner integrals}
\subsection*{Elements of the general theory}
%
%
We summarize two important results of the Bochner theory for functions that are defined on the real semiaxis $\oointerval{0}{\infty}$ and take values in the separable Hilbert space $\mathsf{L}^2\at{\mathbb{R}}$. For the general theory and the proofs we refer to \cite[chapter 2]{R20} and \cite[chapter 1]{HN16}.
\begin{theorem}[Bochner measurability and Bochner integrability]\label{theorem_bochner}
The following statements are satisfied for any function $F: \pair{0}{\infty}\to \mathsf{L}^2\at{\mathbb{R}}$.
\begin{itemize}
\item[1.] 
$F$ is Bochner measurable if and only if $\xi \mapsto \langle F\at{\xi}, G \rangle_2$ is Lebesgue measurable for all $G\in  \mathsf{L}^2\at{\mathbb{R}}$, where
$\langle \cdot, \cdot \rangle_2$ denotes the inner product on $\mathsf{L}^2\at{\mathbb{R}}$.
\item[2.] 
Let $F$ be Bochner measurable. Then $F$ is Bochner integrable if and only if $\xi\mapsto\|F\at{\,\xi\,}\|_2$ is Lebesgue integrable. In this case we have
\begin{align}\notag
\left\| \int_0^\infty F\at{\xi}\dint{\xi}\right \|_2\leq \int_0^\infty \| F\at{\xi}\|_2\dint{\xi}\,.
\end{align}
\end{itemize}
\end{theorem}
\begin{theorem}[Dominated Convergence Theorem]\label{theorem_lebesgue_bochner}
Let $(F_n)_{n\in\mathbb{N}}$ be a sequence of Bochner integrable functions $F_n: (0,\infty)\to \mathsf{L}^2\at{\mathbb{R}}$. 
If there exist a function  $F_\infty: (0,\infty)\to \mathsf{L}^2\at{\mathbb{R}}$ and a Lebesgue integrable function $g:(0,\infty) \to \mathbb{R}$ such that 
\begin{align*}
\|F_n\at{\xi}-F_\infty\at{\xi}\|_2 \xrightarrow{\;n\to\infty\;} 0\,,\qquad \|F_n\at{\xi}\|_2 \leq g\at{\xi} 
\end{align*}
holds for almost all $\xi\in\pair{0}{\infty}$, then $F_\infty$ is Bochner integrable and we have
\begin{align*}
\lim_{n\to\infty} \int_0^\infty \|F_n\at{\xi}-F_\infty\at{\xi}\|_2\dint{\xi}=0
\end{align*}
as well as
\begin{align*}
\lim_{n\to\infty} \int_0^\infty F_n\at{\xi}\dint{\xi}=\int_0^\infty F_\infty\at{\xi}\dint{\xi} \qquad \text{ strongly in } \quad \mathsf{L}^2\at{\mathbb{R}}\,.
\end{align*}
\end{theorem}
%
%
\subsection*{Special results}
%
%
We next show that the integral operators in \eqref{operator_B_eps_peridynamics}, \eqref{operators_Q_eps_P_eps_peridynamics}, and \eqref{operator_M_eps_peridynamic} are in fact well-defined in the sense of Bochner integrals. Afterwards we 
characterize the Fourier transform of $\calB_\eps$.
\begin{proposition}[operators from \S\ref{sect:prelim} and \S\ref{section_main_result}]\label{prop:Bochner1}
The operators \mbox{$\calB_\eps,\,\calQ_\eps,\, \calP_\eps ,\, \calM_\eps : \mathsf{L}^2\at{\mathbb{R}}\to \mathsf{L}^2\at{\mathbb{R}}$} are well-defined  in the sense of Bochner integrals for any $\eps>0$. They satisfy
\begin{align*}
\|\calB_\eps V\|_2\leq \|V\|_2+\int_0^\infty \alpha\at{\xi}\,\xi^2 \frac{\| V-\calA_{\eps \xi}^2V\|_2}{\eps^2}\dint{\xi}
\end{align*}
as well as
\begin{align*}
\|\calQ_\eps [V]\|_2\leq \int_0^\infty \beta\at{\xi}\,\xi^3 \|\calA_{\eps\xi}\left(\calA_{\eps \xi} V\right)^2\|_2\dint{\xi}\,,
\quad
\|\calM_\eps V\|_2\leq 2\int_0^\infty \beta\at{\xi}\,\xi^3 \|\calA_{\eps\xi}\left((\calA_{\eps \xi} W_0)(\calA_{\eps \xi} V)\right)\|_2\dint{\xi}
\end{align*}
and
\begin{align*}
\|\calP_\eps [V]\|_2\leq \frac{1}{\eps^6}\int_0^\infty \xi\,\| \calA_{\eps\xi}\partial_r \psi \pair{\eps^2\,\xi\,\calA_{\eps \xi} V}{\xi}\|_2\dint{\xi}\,.
\end{align*}
Moreover, we have $\eps^2\,\calP_\eps[V]\xrightarrow{\;\eps\to 0\;} 0$ for any $V\in\mathsf{L}^2\at{\mathbb{R}}$.
\end{proposition}
\begin{proof}
We only discuss $\calB_\eps$ and $\calP_\eps$ in detail. The statements concerning $\calQ_\eps$ and $\calM_\eps$ can be established by similar arguments. 
\par\emph{\ul{Operator $\calB_\eps$}}\,: %
For given  $V\in\mathsf{L}^2\at{\mathbb{R}}$ and $\eps>0$ we define $F_\eps : (0,\infty)\to \mathsf{L}^2\at{\mathbb{R}}$ by
\begin{align*}
F_\eps\at{\xi}:=\alpha\at{\xi}\,\xi^2 \frac{V-\calA_{\eps \xi}^2 V}{\eps^2}
\end{align*}
and observe that the function
\begin{align*}
\xi\mapsto\langle F_\eps\at{\xi}, G \rangle_2
= \frac{\alpha\at{\xi}\,\xi^2}{\eps^2}\left(\langle V, G \rangle_2 - \langle \calA_{\eps \xi}^2 V ,G \rangle_2\right)
\end{align*}
is Lebesgue measurable for any $G\in \mathsf{L}^2\at{\mathbb{R}}$ since $\xi\mapsto\alpha\at{\xi}\, \xi^2$ is at least piecewise continuous 
according to Assumption \ref{assumption_1} and because Lemma \ref{properties_A_eta} ensures that $\xi\mapsto \langle V ,G \rangle_2 - \langle \calA_{\eps \xi}^2 V ,G \rangle_2$  is differentiable with respect to $\xi$. The function $F_\eps$ is thus Bochner measurable thanks to Theorem \ref{theorem_bochner}. Moreover, \eqref{estimate_A_eta_1}$_2$ and \eqref{integral_alpha}$_1$ yield the estimate
\begin{align*}
\int_0^\infty \| F_\eps\at{\xi} \|_2 \dint{\xi}
 \leq \int_0^\infty \frac{\alpha\at{\xi}\,\xi^2}{\eps^2}\left(\|V\|_2+\|\calA_{\eps\xi}^2 V\|_2\right)\dint{\xi}
= 2\, \| V\|_2 \frac{1}{\eps^2} \int_0^\infty \alpha\at{\xi}\,\xi^2 \dint{\xi} 
< \infty\,,
\end{align*}
and this implies the Bochner integrability of $F_\eps$ as well as  the desired bound for $\|\calB_\eps V\|_2$.
\par\emph{\ul{Operator $\calP_\eps$}}\,: %
We now define  $F_\eps : (0,\infty)\to \mathsf{L}^2\at{\mathbb{R}}$ by
\begin{align*}
F_\eps\at{\xi}:=\xi\, \calA_{\eps\xi}\partial_r \psi \pair{\eps^2\,\xi\,\calA_{\eps \xi} V}{\xi}\,,
\end{align*}
fix $G\in\mathsf{L}^2\at{\mathbb{R}}$, and write
\begin{align*}
\langle F_\eps\at{\xi}, G \rangle_2 
&=\int_{\mathbb{R}} h_\eps\pair{\xi}{x} \dint{x}\,,
\end{align*}
where  $h_\eps: \pair{0}{\infty}\times\mathbb{R}\to\mathbb{R}$ is given by
\begin{align*}
h_\eps\pair{\xi}{x}:= \xi \, \partial_r \psi \pair{\eps^2\,\xi\,(\calA_{\eps \xi} V)\at{x}}{\xi}\, (\calA_{\eps\xi}G)\at{x}
\end{align*}
since $\calA_{\eps\xi}$ is self-adjoint. This function $h_\eps$ is continuous by Assumption \ref{assumption_1} and  Lemma \ref{properties_A_eta}, so Theorem \ref{theorem_bochner} combined with Fubini's theorem ensures that $F_\eps$ is Bochner measurable. Moreover, \eqref{estimate_second_derivative_psi}, \eqref{estimate_A_eta_1}, and \eqref{estimate_A_eta_pointwise} guarantee that the estimate
\begin{align*}
\babs{\partial_r \psi\pair{\eps^2\,\xi\,(\calA_{\eps\xi}V)}{\xi}}
\leq\gamma\at{\xi}\,\bat{\eps^2\,\xi\,\calA_{\eps\xi}|V|}^3\leq
\ga\at\xi\,\eps^6\,\xi^3\norm{\calA_{\eps\xi}|V|}_\infty^2\calA_{\eps\xi}\abs{V}
\leq
\ga\at\xi\,\eps^5\,\xi^2\norm{V}_2^2\calA_{\eps\xi}\abs{V}
\end{align*}
holds for any $\xi$ in the sense of functions with variable $x$. Using also \eqref{estimate_A_eta_1} we therefore get
\begin{align*}
|F_\eps\at{\xi}|
\leq\gamma\at{\xi}\,\eps^5 \,\xi^3\,\norm{V}_2^2\,\calA_{\eps\xi}\bat{\calA_{\eps\xi}|V|}\,,\qquad
\|F_\eps\at{\xi}\|_2 
&\leq  \gamma\at{\xi}\,\eps^5 \,\xi^3\,\norm{V}_2^3
\end{align*}
and hence
\begin{align*}
\|\calP_\eps [V]\|_2
 \leq {\eps^{-6}}\int_0^\infty  \|F_\eps\at{\xi}\|_2 \dint{\xi}
\leq \eps^{-1}\,\|V\|_2^3\int_0^\infty  \gamma\at{\xi}\,\xi^3\dint{\xi}\,.
\end{align*}
This implies both the desired estimate and the claimed convergence result for $\calP_\eps$ thanks to \eqref{integral_gamma}.
\end{proof}
\begin{proposition}[Fourier symbol of $\calB_\eps$]\label{prop:Bochner2}
For $\eps>0$ and any $V\in\mathsf{L}^2\at{\mathbb{R}}$ we have
\begin{align*}
\widehat{\calB_\eps V}\at{k}=b_\eps\at{k}\,\widehat{V}\at{k}
\end{align*}
for almost all $k\in\Rset$, where the symbol function $b_\eps$ is given in \eqref{symbol_b_eps_peridynamics}.
\end{proposition}
\begin{proof}
\emph{\ul{Special case}}\,: 
For any $V\in\mathsf{L}^1\at{\mathbb{R}} \cap \mathsf{L}^2\at{\mathbb{R}} \cap \mathsf{BC}\at{\mathbb{R}}$ we get
\begin{align*}
\int_{\mathbb{R}}\int_0^\infty \alpha\at{\xi}\,\xi^2\,V\at{x}\,\mathtt{e}^{\mathtt{i}kx}\dint{\xi}\dint{x}=\int_0^\infty \int_{\mathbb{R}}\alpha\at{\xi}\,\xi^2\,V\at{x}\,\mathtt{e}^{\mathtt{i}kx}\dint{x}\dint{\xi}
= \left( \int_{\mathbb{R}}\alpha\at{\xi}\,\xi^2\dint{\xi}\right) \widehat{V}\at{k}
\end{align*}
by Fubini's theorem since the integrand is continuous with respect to both $\xi$ and $x$. The function $h_\eps : \pair{0}{\infty} \times \mathbb{R}\to\mathbb{R}$ with
\begin{align*}
h_\eps\pair{\xi}{x}:=\alpha\at{\xi}\,\xi^2 (\calA_{\eps\xi}^2V)\at{x}\,\mathtt{e}^{\mathtt{i}kx}
\end{align*}
is also continuous for any given $k$ and \eqref{estimate_A_eta_1}$_2$ implies
\begin{align*}
\int_0^\infty \int_\mathbb{R} |h_\eps\pair{\xi}{x}|\dint{x}\dint{\xi}
\leq \|V\|_1  \int_0^\infty \alpha\at{\xi}\,\xi^2\dint{\xi}
<\infty\,.
\end{align*}
Using Fubini's theorem as well as \eqref{symbol_A_eta} we thus obtain
\begin{align*}
\int_\mathbb{R}\int_0^\infty  h_\eps\pair{\xi}{x}\dint{\xi}\dint{x}
&=\int_0^\infty \alpha\at{\xi}\,\xi^2 \int_\mathbb{R}(\calA_{\eps\xi}^2 V) \at{x}\,\mathtt{e}^{\mathtt{i}kx}\dint{x}\dint{\xi}
=\int_0^\infty \alpha\at{\xi}\,\xi^2 \operatorname{sinc}^2\at{\frac{\eps\,\xi\,k}{2}}\,\widehat{V}\at{k}\dint{\xi}\,,
\end{align*}
and hence
\begin{align*}
\widehat{\calB_\eps V}\at{k}
&=\int_{\mathbb{R}}\left(V\at{x}+\int_0^\infty \alpha\at{\xi}\,\xi^2 \,\frac{V\at{x}-(\calA_{\eps\xi}^2 V)\at{x}}{\eps^2}\dint{\xi}\right) \mathtt{e}^{\mathtt{i}kx}\dint{x}\\
&= \widehat{V}\at{k}+\int_0^\infty \alpha\at{\xi}\,\xi^2 \frac{\widehat{V}\at{k}-\operatorname{sinc}^2\at{\frac{\eps\,\xi\,k}{2}}\widehat{V}\at{k}}{\eps^2}\dint{\xi}=b_\eps\at{k}\,\widehat{V}\at{k}
\end{align*}
for all $k\in\mathbb{R}$.
\par\emph{\ul{Approximation argument}}\,: %
Now let $V\in\mathsf{L}^2\at{\mathbb{R}}$ be arbitrary and $(V_n)_{n\in\mathbb{N}}\subset \mathsf{L}^1\at{\mathbb{R}} \cap \mathsf{L}^2\at{\mathbb{R}}\cap \mathsf{BC}\at{\mathbb{R}}$ be an approximating sequence. We already proved that $\widehat{\calB_\eps V_n} = b_\eps\widehat{V}_n$ holds for any $n$ and $\hat{V}_n \xrightarrow{\;n\to\infty\;} \hat{ V}$ is provided by Parseval's theorem. Moreover, Theorem \ref{prop:Bochner1} and \eqref{estimate_A_eta_1}$_2$ ensure
\begin{align*}
\frac{1}{\sqrt{2\,\pi}}\,\|\widehat{\calB_\eps V}_n-\widehat{\calB_\eps V}\|_2
&=\|\calB_\eps V_n-\calB_\eps V\|_2
\leq\left( 1+ \frac{2}{\eps^2}\int_0^\infty \alpha\at{\xi}\,\xi^2\dint{\xi}\right) \|V_n-V\|_2\xrightarrow{\;n\to\infty\;}0\,,
\end{align*}
so the claim follows immediately.
\end{proof}
%
%

\end{document}